\documentclass{amsart}

\usepackage{amsmath}
\usepackage{amsthm}
\usepackage{amssymb}
\usepackage{lscape,xcolor}
\usepackage{graphicx}
\usepackage{mathrsfs}
\usepackage{stmaryrd}
\usepackage{verbatim}
\usepackage[T1]{fontenc}
\usepackage{rotating}
\usepackage{tikz-cd}
\usetikzlibrary{decorations.pathmorphing}
\usepackage{breakurl}

\usepackage{hyperref}
\usepackage{euscript}
\usepackage[colorinlistoftodos]{todonotes}
\usepackage{enumitem}
\usepackage[nameinlink]{cleveref}
\usepackage{comment}
\usepackage{soul}
\tikzset{
	symbol/.style={
		draw=none,
		every to/.append style={
			edge node={node [sloped, allow upside down, auto=false]{$#1$}}}
	}
}
\usepackage[nocompress,noadjust]{cite}
\usepackage{xcolor}
\definecolor{seagreen}{RGB}{46,139,87}
\definecolor{maroon}{RGB}{128,0,0}
\definecolor{darkviolet}{RGB}{148,0,211}
\definecolor{twelve}{RGB}{100,100,170}
\definecolor{thirteen}{RGB}{100,150,50}
\definecolor{fourteen}{RGB}{200,0,0}
\definecolor{fifteen}{RGB}{0,200,0}
\definecolor{sixteen}{RGB}{0,0,200}
\definecolor{seventeen}{RGB}{200,0,200}
\definecolor{eighteen}{RGB}{0,200,200}

\allowdisplaybreaks[1]
\calclayout

\newcommand{\mmod}{\! \sslash \!}

\newcommand{\mr}[1]{\mathrm{#1}}
\newcommand{\mbf}[1]{\mathbf{#1}}

\newcommand{\br}[1]{\overline{#1}}

\newcommand{\Z}{\mathbf{Z}}

\newcommand{\Q}{\mathbf{Q}}
\newcommand{\W}{\mathbf{W}}
\newcommand{\F}{\mathbf{F}}
\newcommand{\G}{\mathbf{G}}

\def \HF2{\mr{H}\F_2}

\DeclareMathOperator{\Ext}{Ext}

\DeclareMathOperator{\Map}{Map}

\DeclareMathOperator*{\colim}{colim}

\DeclareMathOperator{\Pic}{Pic}

\def \AA0{\br{A \mmod A(0)}_*}
\def \AA2{A\mmod A(2)_*}
\def \AE2{(A\mmod E(2))_*}
\renewcommand{\AE}[1]{(A\mmod E(#1))_*}

\def \E2E1{(E(2)\mmod E(1))_*}

\newcommand{\Sp}{\mathsf{Sp}}

\newcommand{\Gal}{\mathrm{Gal}}
\newcommand{\simto}{\overset{\sim}{\longrightarrow}}
\newcommand{\Zp}{\Z_p}
\newcommand{\Zpx}{\Zp^\x}

\newcommand{\Qp}{\Q_p}

\newcommand{\x}{\times}
\newcommand{\llb}{\llbracket}
\newcommand{\rrb}{\rrbracket}
\newcommand{\Sbf}{\mathbf{S}}
\newcommand{\ldetr}{\langle\det\rangle}
\newcommand{\cd}{\mathrm{cd}}
\newcommand{\mfrak}{\mathfrak{m}}
\newcommand{\hemail}[1]{\email{\href{mailto:#1}{#1}}}
\newcommand{\ev}{\mathrm{ev}}

\newcommand{\coker}{\mathrm{coker}}
\newtheorem{thm}[equation]{Theorem}
\newtheorem{cor}[equation]{Corollary}
\newtheorem{lem}[equation]{Lemma}
\newtheorem{prop}[equation]{Proposition}

\newtheorem{rem}[equation]{Remark}

\newtheorem*{thm*}{Theorem}
\newtheorem*{cor*}{Corollary}
\newtheorem*{lem*}{Lemma}
\newtheorem*{prop*}{Proposition}
\newtheorem*{not*}{Notation}

\theoremstyle{definition}
\newtheorem{defn}[equation]{Definition}

\newtheorem{rmk}[equation]{Remark}

\newtheorem{question}[equation]{Question}
\newtheorem{conjecture}[equation]{Conjecture}
\newtheorem*{conjecture*}{Conjecture}

\newtheorem{construction}[equation]{Construction}

\newtheorem*{defn*}{Definition}
\newtheorem*{ex*}{Example}
\newtheorem*{exs*}{Examples}
\newtheorem*{rmk*}{Remark}
\newtheorem*{claim*}{Claim}

\numberwithin{equation}{section}
\numberwithin{figure}{section}

\setlist{leftmargin=*}

\title[Exotic Picard groups and chromatic vanishing via the Gross-Hopkins duality]{Exotic Picard groups and chromatic vanishing\protect\\ via the Gross-Hopkins duality}
\author{Dominic Leon Culver}\address{~}
	\hemail{dominic.culver@gmail.com}
\author{Ningchuan Zhang}\address{Indiana University Bloomington, Bloomington, IN 47405, USA}\hemail{nz7@iu.edu}
\begin{document}
	\begin{abstract}
		In this paper, we study the exotic $K(h)$-local Picard groups $\kappa_h$ when $2p-1=h^2$ and the homological Chromatic Vanishing Conjecture when $p-1$ does not divide $h$. The main idea is to use the Gross-Hopkins duality to relate both questions to certain Greek letter element computations in chromatic homotopy theory. Classical results of Miller-Ravenel-Wilson then imply that an exotic element at height $3$ and prime $5$ is not detected by the type-$2$ complex $V(1)$.  For the homological Vanishing Conjecture, we prove it holds modulo the invariant prime ideal $I_{h-1}$. We further show that this special case of the Vanishing Conjecture implies the exotic Picard group $\kappa_h$ is zero at height $3$ and prime $5$. Both results can be thought of as a first step towards proving the vanishing of $\kappa_3$ at prime $5$.
		
	\end{abstract}
	\maketitle
	\setcounter{section}{-1}
	\section{Introduction}
	\subsection{Statement of main results}
	The study of Picard groups in chromatic homotopy theory was initiated by Hopkins in \cite{HMS_picard,Strickland_interpolation}. By analyzing the homotopy fixed point spectral sequence for the $K(h)$-local sphere, Hopkins-Mahowald-Sadofsky proved the following:
	\begin{thm*}[{\cite[Proposition 7.5]{HMS_picard}}]
		The exotic $K(h)$-local Picard group $\kappa_h$ (see \Cref{defn:exotic}) is zero when $p-1$ does not divide $h$ and $2p-1>h^2$.
	\end{thm*}
	In this paper, we study $\kappa_h$ when $2p-1=h^2$. The smallest of such pairs is $h=3$ and $p=5$. Notice that this assumption already implies $(p-1)\nmid h$.
	\begin{rmk*}
		It is an open question in number theory whether there are infinitely primes $p$ such that $2p-1$ is a perfect square (\cite[page 171]{Iwaniec1978}). Using SageMath \cite{sagemath}, the authors are able to find $35,528,083$ positive integers $h$ less than $10^9$ such that $\frac{h^2+1}{2}$ is a prime number.
	\end{rmk*}
	Our first main result is: 
	\begin{thm*}[A, {\Cref{thm:main_A}, \Cref{cor:fin_complex}}]
		Let $2p-1=h^2$. Suppose the type-$(h-1)$ Smith-Toda complex $V(h-2)=S^0/(p,v_1,\cdots, v_{h-2})$ exists at prime $p$. Then an exotic element $X\in \kappa_h$ cannot be detected by $V(h-2)$, i.e. 
		\begin{equation*}
			L_{K(h)}\left(X\wedge V(h-2)\right)\simeq L_{K(h)}V(h-2).
		\end{equation*}
		In particular,
		\begin{enumerate}
			\item At height $3$ and prime $5$, an exotic element $X$ in $\Pic_{K(3)}$ cannot be detected by $V(1)=S^0/(5,v_1)$.
			\item  At height $5$ and prime $13$, an exotic element $X$ in $\Pic_{K(5)}$ cannot be detected by $V(3)=S^0/(13,v_1,v_2,v_3)$.
		\end{enumerate} 
	\end{thm*}
	When $4p-3=h^2$, we prove a similar statement  in \Cref{thm:kappa_h_4p-3} for a subgroup $\kappa_h^{(1)}$ of the exotic Picard group $\kappa_h$ defined in \Cref{subsec:Pic_filtration}.  In particular at $(h,p)=(3,3)$ and $(5,7)$, we show that $V(h-2)$ cannot detect elements in this subgroup of $\kappa_h$. 
	
	Our method is also used to study the following special case of the Chromatic Vanishing Conjecture (\ref{conj:CVC}), first proposed in \cite{Beaudry_orbits,Beaudry-Goerss-Henn}. 
	\begin{conjecture*}[Reduced Homological Vanishing Conjecture, (RHVC)]
		\[\F_p\cong H_0(\G_h;\F_{p^h})\simto H_0(\G_h;\pi_0(E_h)/p). \]
	\end{conjecture*}
	\begin{rmk*}
		The Vanishing Conjecture was stated in terms of group \emph{co}homology in \cite[Conjecture 1.1.4]{Beaudry-Goerss-Henn}. This is equivalent to the homological versions when $(p-1)\nmid h$ by Poincar\'e duality. See \Cref{rem:vanishing_conj}. 
	\end{rmk*}
	\begin{thm*}[B, {\Cref{thm:RHVC_Ih-2}}]
		When $(p-1)\nmid h$, the RHVC holds modulo the ideal $I_{h-1}=(p,u_1,\cdots, u_{h-2})$, i.e. there  are isomorphisms:
		\[\F_p\cong H_0(\G_h;\F_{p^h})\simto H_0(\G_h;\pi_0(E_h)/I_{h-1}). \]
	\end{thm*}
	\noindent Exotic Picard groups and the Vanishing Conjecture are related by:
	\begin{thm*}[C, {\Cref{thm:main_C}}]
		If the RHVC holds at height $3$, then $\kappa_3=0$ at $p=5$ and $\kappa_3^{(1)}=0$ at $p=3$, where $\kappa_3^{(1)}$ is a subgroup of $\kappa_3$ defined in \Cref{subsec:Pic_filtration}
	\end{thm*}
	For general heights and primes, we give some bounds on the divisibility of Greek letter elements that would imply the RHVC (when $(p-1)\nmid h$) and $\kappa_h=0$ (when $2p-1=h^2$) in \Cref{prop:bounds_RHVC_kappah}.
	\subsection{General strategy}\label{subsec:strategy}
	A summary of our strategy to study exotic Picard groups when $2p-1=h^2$ is as follows. We will show successively each claim below is implied by the following one. 
	\begin{enumerate}[label=\Roman*.,itemsep=1 ex]
		\item $\kappa_h=0$.
		\item $H^{h^2}_c(\Sbf_h;\pi_{2p-2}(E_h))=H^{2p-1}_c(\Sbf_h;\pi_{2p-2}(E_h))=0$.
		\item $H^{h^2}_c(\Sbf_h;\pi_{2p-2}(E_h)/p)=0$.
		\item $H^{0}_c(\Sbf_h;\pi_{2h-2p+2}(E_h)\ldetr/(p,u_1^\infty,\cdots,u_{h-1}^\infty))=0$, where the determinant twist $\ldetr$ is defined in \Cref{defn:det} and the quotient mod $(p,u_1^\infty,\cdots,u_{h-1}^\infty) $ is explained in \Cref{defn:mod_m_infty}. 
		\item $H^0_c\left(\Sbf_h;\pi_{2h-2p+2-\frac{p^N|v_h|}{p-1}}(E_h)/J\right)=0$ for any open invariant ideal $J\trianglelefteq \pi_0(E_h)$ containing $p$ such that $v_{h}^{p^N}$ is invariant mod $J$.
		\item $\Ext_{BP_*BP}^{0,2h-2p+2-\frac{p^N|v_h|}{p-1}}(BP_*,v_h^{-1}BP_*/J)=0$ for any invariant ideal $J\trianglelefteq v_{h}^{-1}BP_*$ containing $p$ such that $v_{h}^{p^N}$ is invariant mod $J$.
		\item $H^{0,t}(M^{h-1}_1)=0$ for any $t\equiv 2h-2p+2-\frac{p^N|v_h|}{p-1}\mod p^N|v_h|$ and all integers $N\ge 0$, where $M^{h-1}_1:=v_h^{-1}BP_*/(p,v_1^{\infty},\cdots,v_{h-1}^{\infty})$.
	\end{enumerate}	
	\textbf{II$\implies$I}:\quad In \cite{GHMR_Picard}, Goerss-Henn-Mahowald-Rezk defined a map that detects the exotic Picard group $\kappa_h$:
	\[\ev_2\colon \kappa_h\to H^{2p-1}_c(\G_h;\pi_{2p-2}(E_h)). \]
	Using the same argument as in \cite{HMS_picard}, we will show this map is injective when $(p-1)\nmid h$ and $4p-3>h^2$ in \Cref{prop:ev_inj}.\footnote{A descent spectral sequence for $K(h)$-local Picard groups in \cite[Example 6.18]{Heard_2021Sp_kn-local} implies this map is an isomorphism under the assumptions. See \Cref{prop:ev_surj}.} As a result, $\kappa_h$ vanishes if $H^{2p-1}_c(\G_h;\pi_{2p-2}(E_h))=0$ when $2p-1=h^2$. By \cite[Lemma 1.32]{Bobkova-Goerss} and \cite[page 12]{Goerss-Hopkins}, we have
	\[H^{s}_c(\G_h;\pi_{t}(E_h))\cong H^{s}_c(\Sbf_h;\pi_{t}(E_h))^{\Gal}\text{ for any $s$ and $t$,} \]	
	where $\Sbf_h\le \G_h$ is the automorphism group of the height $h$-Honda formal group. This indicates we just need to show the relevant group cohomology of $\Sbf_h$ is zero.\\ 
	\textbf{III$\implies$II}:\quad	Now suppose $2p-1=h^2$. By \Cref{thm:lazard} of Lazard and the fact $\Sbf_h$ has no finite $p$-group, $\cd_p(\Sbf_h)=h^2$. When $(p-1)\nmid h$, the cohomology we are computing $H^{2p-1}_c(\G_h;\pi_{2p-2}(E_h))=H^{h^2}_c(\G_h;\pi_{2p-2}(E_h))$ is a top degree cohomology. Using a Hochschild-Lyndon-Serre spectral sequence and the explicit formula of the action by the center $\Zpx$ of $\Sbf_h$, we show in \Cref{prop:Hh2_mod_p} that
	\[H^{h^2}_c(\G_h;\pi_{2p-2}(E_h))\simto H^{h^2}_c(\G_h;\pi_{2p-2}(E_h)/p). \]
	Alternatively, the above isomorphism can be proved using the Poincar\'e duality between top degree cohomology and zero degree homology. \\
	\textbf{IV$\implies$III}:\quad There is another Poincar\'e duality between top and zero degree cohomology groups for any $p$-complete $\G_h$-module $M$:
	\[H^{h^2}_c(\Sbf_h;M)\cong H^{0}_c(\Sbf_h;M^\vee)^\vee,  \]
	where $(-)^\vee:=\hom_c(-,\Qp/\Zp)$ is the continuous equivariant Pontryagin dual (\Cref{defn:Pontryagin_dual}). For $M=\pi_t(E_h)$, the dual $M^\vee$ is identified by Gross-Hopkins duality \Cref{cor:GH_dual}:
	\[\pi_t(E_h)^\vee\cong \pi_{2h-t}(E_h)\ldetr/\mfrak^\infty, \]
	where $\mfrak=(p,u_1,\cdots,u_{h-1})\trianglelefteq \pi_0(E_h)$ is the maximal ideal, mod $\mfrak^\infty$ is defined in \Cref{defn:mod_m_infty}, and $\ldetr$ is the determinant twist  defined in \Cref{defn:det}). In the case when $t=2p-2$, we further have:
	\begin{align*}
		H^{h^2}_c(\Sbf_h;\pi_{2p-2}(E_h))&\cong H^{h^2}_c(\Sbf_h;\pi_{2p-2}(E_h)/p)\\&\cong H_c^0(\Sbf_h;\pi_{2h-2p+2}(E_h)\ldetr/(p,u_1^\infty,\cdots,u_{h-1}^\infty))^\vee .
	\end{align*}
	\textbf{V$\implies$IV}:\quad In \cite{Hopkins_1994}, Gross-Hopkins identified the determinant twist mod $p>2$ with a limit of finite suspensions:
	\[\pi_0(E_h)\ldetr/p\cong \Sigma^{\lim\limits_{N\to \infty}\frac{p^N|v_h|}{p-1}}\pi_0(E_h)/p. \]
	This is a limit in the \emph{algebraic} $K(h)$-local Picard group. More precisely, let $J\trianglelefteq \pi_0(E_h)$ be an open invariant ideal containing $p$, such that $v_h^{p^N}$ is invariant modulo $J$. Then \[\pi_0(E_h)\ldetr/J \cong \Sigma^{\frac{{p^N|v_h|}}{p-1}}\pi_0(E_h)/J. \]
	
	By \Cref{prop:Hh2_reduction}, we now have 
	\begin{align*}
		&H_c^0(\Sbf_h;\pi_{2h-2p+2}(E_h)\ldetr/(p,u_1^\infty,\cdots,u_{h-1}^\infty))\\\cong&\colim_{p\in J\trianglelefteq \pi_0(E_h)}H^0_c\left(\Sbf_h;\left.\pi_{2h-2p+2-\frac{p^N|v_h|}{p-1}}(E_h)\right/J\right).
	\end{align*}
	As a result, to show the left hand side is zero, it suffices to show every single term in the colimit system on right hand side is zero. \\
	\textbf{VI$\implies$V}\quad Using a Change of Rings theorem, \Cref{thm:CoR}, we relate the group cohomology of $\G_h$ with Ext-groups of $BP_*BP$-comodules:
	\[H^s(\G_h;\pi_t(E_h)/J)\cong \Ext_{BP_*BP}^{s,t}(BP_*,v_h^{-1}BP_*/J') \]
	for some invariant ideal $J'\trianglelefteq v_h^{-1}BP_*$. When $J=(p,u_1^{j_1},\cdots, u_{h-1}^{j_{h-1}})$, we can take $J'=(p,v_1^{j_1},\cdots, v_{h-1}^{j_{h-1}})$. As a result, we need to compute $\Ext_{BP_*BP}^{0,t}(BP_*,v_h^{-1}BP_*/J')$ for certain values of $t$.\\
	\textbf{VII$\implies$VI}\quad  For a $BP_*BP$-comodule $M$, we denote $\Ext_{BP_*BP}^{s,t}(BP_*,M)$ by $H^{s,t}(M)$. The colimit of the cohomology groups $H^{0,t}(v_h^{-1}BP_*/J)$ over all invariant ideals $J\trianglelefteq v_h^{-1}BP_*$ containing $p$ is $H^{0,t}(M^{h-1}_1)$, where $M_{1}^{h-1}=v_h^{-1}BP_*/(p,v_1^\infty,\cdots,v_{h-1}^\infty)$. This is the group of mod-$p$ Greek letter elements at height $h$. Keeping track of the degree $t$, we have reduced our computation to the following:
	\begin{prop*}
		Suppose $2p-1=h^2$. If $H^{0,t}(M^{h-1}_1)=0$ whenever $t\equiv 2h-2p+2-\frac{p^N|v_h|}{p-1}\mod p^N|v_h|$ for some integer $N\ge 0$, then $\kappa_h=0$. 
	\end{prop*}
	The argument above can also be used to study the Chromatic Vanishing Conjecture (\ref{conj:CVC}) in degree $0$ homology groups when $(p-1)\nmid h$. This conjecture has been verified at all primes at heights $1$ and $2$ by explicit computations. It plays an essential role in Beaudry-Goerss-Henn's works in \cite{Beaudry-Goerss-Henn} to disprove and completely understand the Chromatic Splitting Conjecture at $h=p=2$. The Vanishing Conjecture is wide open at $h\ge 3$. Using Gross-Hopkins duality and Change of Rings theorem, we can translate the Reduced Homological Vanishing Conjecture \eqref{eqn:RHVC} to Greek letter element computations:  
	\begin{prop*}
		Suppose $p-1$ does not divide $h$. If $H^{0,t}(M^{h-1}_1)=\F_p$ whenever $t\equiv 2h-\frac{p^N|v_h|}{p-1}\mod p^N|v_h|$ for some integer $N\ge 0$, then $H_0(\G_h;\pi_0(E_h)/p)=\F_p$ and the RHVC holds.
	\end{prop*}
	\subsection{Greek letter element computations}
	Next, we need to compute the Greek letter elements in $H^{0,t}(M^{h-1}_1)$. Elements in this group are classified into three families in \Cref{prop:families}. 
	\begin{enumerate}
		\item Family I elements are of the form $\frac{v^s_h}{pv_1\cdots v_{h-1}}$, where $(s,p)=1$. In \Cref{prop:fam_I}, we prove Family I elements contribute to a copy $\F_p$ in $H^{h^2}_c(\G_h;\pi_0(E_h)/p)$ via Gross-Hopkins duality, which is predicted in the RHVC. This family does not contribute to $H^{h^2}_c(\G_h;\pi_{2p-2}(E_h)/p)$.
		\item Family II elements are of the form $\frac{1}{pv_1^{d_1}\cdots v_{h-1}^{d_{h-1}}}$, where $(p,v_1^{d_1},\cdots,v_{h-1}^{d_{h-1}})$ is an invariant ideal. In \Cref{cor:fam_II}, we show this family does not contribute to either $H^{h^2}_c(\G_h;\pi_0(E_h)/p)$ or $H^{h^2}_c(\G_h;\pi_{2p-2}(E_h)/p)$.
		\item Family III elements are of the form $\frac{y^s_{h,N}}{pv_1^{d_1}\cdots v_{h-1}^{d_{h-1}}}$, where $y_{h,N}$ is some replacement of $v_h^{p^N}$, $(s,p)=1$ and $(p,v_1^{d_1},\cdots,v_{h-1}^{d_{h-1}}, y^s_{h,N})$ is an invariant regular ideal. While the precise conditions on the $d_i$'s are out of reach in the general situation, we established some bounds in \Cref{prop:fam_III} which would imply this family does not contribute to either $H^{h^2}_c(\G_h;\pi_0(E_h)/p)$ or $H^{h^2}_c(\G_h;\pi_{2p-2}(E_h)/p)$.
	\end{enumerate}
	Combining the three cases above, we obtain the bounds on divisibility of Greek letter elements that would imply the RHVC (when $(p-1)\nmid h$) and vanishing of $\kappa_h$ (when $2p-1=h^2$) in \Cref{prop:bounds_RHVC_kappah}. 
	
	In \cite{MRW}, Miller-Ravenel-Wilson computed $H^{0,*}(M^1_{h-1})$, where $M_{h-1}^1:=v_h^{-1}BP_*/(p,v_1,\cdots,v_{h-2},v_{h-1}^{\infty})$. Using Gross-Hopkins duality and Morava's Change of Rings Theorem, the Miller-Ravenel-Wilson computation yields when $(p-1)\nmid h$, \begin{align*}
		H^{h^2}_c(\G_h;\pi_{0}(E_h)/I_{h-1})&=\F_p,\\
		H^{h^2}_c(\G_h;\pi_{2p-2}(E_h)/I_{h-1})&=0.
	\end{align*} 
	It follows from first isomorphism that the RHVC holds modulo the ideal $I_{h-2}=(p,u_1,\cdots,u_{h-2})\trianglelefteq \pi_0(E_h)$. This is the statement of Main Theorem B \ref{thm:RHVC_Ih-2}.  The second group cohomology measures if there is an exotic element in $\Pic_{K(h)}$ detected by the type-$(h-1)$ Smith-Toda complex $V(h-2):=S^0/(p,v_1,\cdots,v_{h-2})$, provided the latter exists. Consequently, its vanishing yields Theorem A (\ref{thm:main_A}). At height $3$ and prime $5$, we further show in Theorem C (\ref{thm:main_C}) that the RHVC implies $\kappa_3=0$. This proof relies on the Miller-Ravenel-Wilson results. 
	\begin{rmk*}[{\ref{rem:ST_cpx} and \ref{rem:fin_cpx}}]
		We learned from a referee that it is an open question whether $V(h)$ exists when $h\ge 4$ at \emph{any} prime. By \cite[Corollary 7.11]{Hovey-Strickland_1999}, if $X\wedge_{K(h)} V\simeq V$ for all $X\in \kappa_h$ and finite complexes $V$ of type $n$, then $\kappa_h=0$. Main Theorem A (\ref{thm:main_A}) can therefore be thought of as a first step towards showing $\kappa_h=0$ when $2p-1=h^2$, since it implies $X\wedge_{K(h)} V$ for any cofibers $V$ of $v_h$-self maps of $V(h-2)$. Our choices of finite complexes are restricted to cofibers of the Smith-Toda complex $V(h-2)$, because we do not have better Greek letter element computations beyond $H^0(M^{1}_{h-1})$ in \cite{MRW} when $h\ge 3$. 
	\end{rmk*}
	\subsection{Notations and Conventions}
	
	Throughout, we will let $E_h$ denote a fixed Morava $E$-theory based on a height $h$ formal group, typically the height $h$ Honda formal group $\Gamma_h$. For a $K(h)$-local spectrum $X$, we will write $(E_h)_*X$ for the completed $E_h$-homology of $X$. That is, we write 
	\[
	(E_h)_*X:= \pi_*(L_{K(h)}(E_h\wedge X)). 
	\]
	We will also write $X\wedge_{K(h)}Y$ for the $K(h)$-local smash product $L_{K(h)}(X\wedge Y)$.
	
	Denote by $\W:=\W\F_{p^h}$ the ring of Witt vectors over $\F_{p^h}$. We will write $\Sbf_h$ for the Morava stabilizer group, i.e. the automorphisms of a $\Gamma_h$, and we will write $\G_h$ for the extended Morava stabilizer group.
	
	\subsection{Acknowledgments}
	The authors would like to thank Matt Ando, Agn\`es Beaudry, Mark Behrens, Paul Goerss, Hans-Werner Henn,  Guchuan Li, Doug Ravenel, and Vesna Stojanoska for helpful discussions related to this work. We would also like to thank the anonymous referees  for their comments and suggestions on revisions. Some of this work was done while the first author was at the Max-Planck-Institute for Mathematics, and the second author was at University of Illinois Urbana-Champaign as a visiting scholar. We would like to thank both institutes for their hospitality and support. 
	
	N. Zhang was partially supported by NSF Grant DMS-2304719 during the revision of this work. 
	\section{The $K(h)$-local Picard group}
	\subsection{Definitions}	
	In chromatic homotopy theory, we study the stable homotopy category of spectra $\Sp$ via the height filtration of the moduli stack of formal groups at each prime $p$. One such layer in this filtration is the category of $K(h)$-local spectra $\Sp_{K(h)}$, where $K(h)$ is the Morava $K$-theory at $h$ and prime $p$. Like $\Sp$, the category $\Sp_{K(h)}$ also has a symmetric monoidal structure 
	\[X\wedge_{K(h)}Y:=L_{K(h)}(X\wedge Y).\]
	For $\Sp$, its Picard group is given by
	\begin{thm}[{\cite[page 90]{HMS_picard}}]
		The map $\Z\to \Pic(\Sp), n\mapsto S^n$ is an isomorphism of groups. 
	\end{thm}
	The Picard group $\Pic_{K(h)}$ for $\Sp_{K(h)}$, however, is still not fully understood. Here we give a filtration on $\Pic_{K(h)}$ via a sequence of algebraic detection maps $\ev_i$. The first fact is:
	\begin{thm}[{\cite[Theorem 1.3]{HMS_picard}}]\label{thm: E homology invertible}
		The followings are equivalent:
		\begin{itemize}
			\item $X\in \Sp_{K(h)}$ is invertible.
			\item $(E_h)_*(X)$ is an invertible graded $(E_h)_*$-module. 
		\end{itemize}
	\end{thm}
	As $E_h$ is even periodic, an invertible graded $(E_h)_*$-module is either itself or its suspension. This yields the zeroth detection map:
	\[\ev_0\colon \Pic_{K(h)}\xrightarrow{X\mapsto (E_h)_*(X)} \Pic(\text{graded }{(E_h)_*}\text{-modules})= \Z/2.\]
	\begin{prop}
		$\ev_0$ is a surjective group homomorphism. 
	\end{prop}
	\begin{proof}
		We can check $\ev_0$ is a group homomorphism using the K\"{u}nneth theorem. It is surjective since $\ev_0(S^1)=\pi_*(\Sigma E_h)$ is concentrated in odd degrees.
	\end{proof}
	Denote the kernel of $\ev_0$ by $\Pic_{K(h)}^0$. This is the group of invertible $K(h)$-local spectra whose $E_h$-homology is concentrated in even degrees. For any spectrum $X$, its $E_h$-homology is not only a graded $(E_h)_*$-module, but also a \emph{graded $\pi_*(E_h\wedge_{K(h)} E_h)$-comodule}. In the case when $X\in \Pic_{K(h)}^0$, this \emph{graded} comodule structure is determined by $(E_h)_0(X)$ as an \emph{ungraded} $\pi_0(E_h\wedge_{K(h)} E_h)$-comodule. This gives rise to the first detection map:
	\[\ev_1\colon \Pic_{K(h)}^0\xrightarrow{X\mapsto (E_h)_0(X)}\Pic((\pi_0(E_h),\pi_0(E_h\wedge_{K(h)} E_h))\text{-comodules}).\]
	To identify the target of $\ev_1$, we use the following lemma.
	\begin{lem}[{\cite{hovey2004}}]
		There is an isomorphism of Hopf algebroids:
		\[ (\pi_0(E_h),\pi_0(E_h\wedge_{K(h)} E_h))\cong (\pi_0(E_h),\Map_c(\G_h;\pi_0(E_h))), \]
		where $\G_h=\Sbf_h\rtimes\Gal(\F_{p^h}/\F_p)$ and $\Sbf_h$ is the automorphism group of the height-$h$ Honda formal group.
	\end{lem}
	It follows that a $\pi_0(E_h\wedge_{K(h)} E_h)$-comodule $M$ is equivalent to a $\pi_0(E_h)$-module together with a \emph{continuous} $\G_h$-action such that the following diagram commutes for all $g\in \G_h$: (\cite[page 118]{HMS_picard})
	\begin{equation*}
		\begin{tikzcd}
			\pi_0(E_h)\otimes M\rar["g\otimes g"]\dar& \pi_0(E_h)\otimes M\dar\\
			M\rar["g"]&M
		\end{tikzcd}
	\end{equation*}
	The Picard group of such $\G_h$-$\pi_0(E_h)$-modules is computed by a continuous group cohomology of $\G_h$:
	\begin{prop}[{\cite[Proposition 8.4]{HMS_picard}}]\label{prop:Pic_alg}
		\[\Pic(\text{continuous }\G_h\text{-}\pi_0(E_h)\text{-modules})\cong H^1_c(\G_h;\pi_0(E_h)^\x).\]
	\end{prop}	
	As a result, the first detection map is a group homomorphism:
	\begin{equation}\label{eqn:ev1}
		\ev_1\colon \Pic_{K(h)}^0\to H^1_c(\G_h;\pi_0(E_h)^\x).
	\end{equation}
	\begin{defn}
		The Picard group of \emph{graded} $\G_h$-$(E_h)_*$-modules is called the \textbf{algebraic $K(h)$-local Picard group}, denoted by $\Pic^{alg}_{K(h)}$. The Picard group of \emph{ungraded} $\G_h$-$\pi_0(E_h)$-modules is denoted by $\Pic^{alg,0}_{K(h)}$.
	\end{defn}
	Thus, by \Cref{prop:Pic_alg}, we have
	\[
	\Pic^{alg,0}_{K(h)}=H^1_c(\G_h;\pi_0(E_h)^\x).
	\]
	The first detection map $\ev_1$ then extends to the full Picard group $\Pic_{K(h)}$, which we will also denote by $\ev_1$.
	\begin{prop}
		The $K(h)$-local Picard groups we have introduced so far are related by a map of short exact sequences: 
		\[\begin{tikzcd}
			0\rar&\Pic^{0}_{K(h)}\rar\dar["{\ev_1}"]&\Pic_{K(h)}\rar\dar["\ev_1"]&\Z/2\dar[equal]\rar&0\\
			0\rar&\Pic^{alg,0}_{K(h)}\rar&\Pic^{alg}_{K(h)}\rar&\Z/2\rar&0
		\end{tikzcd}\]
	\end{prop}
	\begin{rem}
		It is known that the short exact sequences do not split at height $h=1$ for all primes \cite{HMS_picard}, and at height $2$ for $p\ge 3$ \cite{GHMR_Picard}.
	\end{rem}
	\begin{cor}\label{cor:ev1_ker}
		The two $\ev_1$ maps in the diagram above have isomorphic kernels and cokernels.
	\end{cor}
	This corollary justifies the usage of $\ev_1$ for both detection maps.
	
	\subsection{Exotic Picard groups}
	Now the question turns to whether $\ev_1$ is injective or surjective. The surjectivity problem is hard and involves obstruction theory. In certain cases, we can show $\ev_1$ is injective.
	\begin{defn}\label{defn:exotic}
		The \textbf{exotic $K(h)$-local Picard group} $\kappa_h$ is the kernel of $\ev_1$ in \eqref{eqn:ev1}.
	\end{defn}
	\begin{thm}[\mbox{\cite[Proposition 7.5]{HMS_picard}}] \label{thm:h2<2p-1}
		The exotic Picard group $\kappa_h$ vanishes when $(p-1)\nmid h$ and $2p-1>h^2$.
	\end{thm}
	The detection of elements in $\kappa_h$ lies in the \textbf{homotopy fixed point spectral sequence} (HFPSS) to compute the $\pi_*(X)$ for $X\in\Sp_{K(h)}$:
	\begin{equation}\label{eqn:HFPSS_X}
		E_2^{s,t}=H_c^s(\G_h;(E_h)_t(X))\Longrightarrow \pi_{t-s}\left(X\right).
	\end{equation}
	For any $X\in \kappa_h$, the $E_2$-page of the HFPSS to compute its homotopy groups is isomorphic to as that for $S^0_{K(h)}$. The potential differences between the two spectral sequences are the higher differentials. We will show that the higher differentials are necessarily zero under the assumption $2p-1>h^2$ and $(p-1)\nmid h$. To see this, we need the following basic facts about the HFPSS: 
	\begin{lem}[\mbox{\cite[Lemma 1.32]{Bobkova-Goerss},\cite[Page 12]{Goerss-Hopkins}}]
		For any $\G_h$-$\pi_0(E_h)$-module $M$, we have an isomorphism $H^s_c(\G_h;M)\cong H^s_c(\Sbf_h;M)^{\Gal}$.
	\end{lem}
	\begin{lem}[Sparseness, \mbox{\cite[Remark 1.4]{Goerss-Hopkins}}]\label{lem:sparseness}
		The continuous group cohomology $H_c^s(\Sbf_h;\pi_t(E_h))$ is zero unless $2(p-1)$ divides $t$.
	\end{lem}
	
	\begin{lem}[Horizontal vanishing line, \mbox{\cite[Proposition 1.6]{Goerss-Hopkins}}]\label{lem:cd}
		The $p$-adic Lie group $\mbf{S}_h$ has cohomological dimension $h^2$ if $(p-1)\nmid h$.
	\end{lem}
	
	It follows that the HFPSS \eqref{eqn:HFPSS_X} has a horizontal vanishing line at $s=h^2$ when $(p-1)\nmid h$.
	\begin{lem}[$0$-line, {\cite[Lemma 1.33]{Bobkova-Goerss}}]\label{lem:HPFSS_s=0}
		$H_c^0(\G_h;\pi_t(E_h))=\left\{\begin{array}{cl}
			\Zp,& t=0;\\
			0,&\textup{otherwise.}
		\end{array}\right.$
	\end{lem}
	
	\begin{proof}[Proof of \Cref{thm:h2<2p-1}]
		We need to show that when $(p-1)\nmid h$ and $h^2<2p-1$, a $K(h)$-local spectrum $X$ is weakly equivalent to $S^{0}_{K(h)}$ if there is a $\G_h$-equivariant isomorphism $(E_h)_*(X)\cong (E_h)_*$.
		
		Under this assumption, HFPSS for $X$ collapses at $E_2$-page by sparseness (\Cref{lem:sparseness}). As a result, any unit $[\iota_X]\in E_2^{0,0}(X)=\Zp$ is a permanent cycle and induces a map $S^0\to X$. This map factors as $S^0\to S^0_{K(h)}\xrightarrow{\iota_X} X$ since $X$ is $K(h)$-local. As $\iota_X\colon S^0_{K(h)}\to X$ induces an isomorphism on the $E_2$-page of the HFPSS, it is a weak equivalence by \cite[Theorem 5.3]{Boardman_ccss}.
	\end{proof}	
	In the general case, the first possible non-trivial differential in \eqref{eqn:HFPSS_X} for $X\in \kappa_h$ is $d_{2p-1}$. Let's consider the possible $d_{2p-1}$-differentials supported by $E_{2p-1}^{0,0}(X)=E_2^{0,0}(X)=\Zp$. 	
	\begin{construction}[{\cite[Construction 3.2]{GHMR_Picard}}] \label{con:exotic_hom}
		Fix an $\G_h$-equivariant isomorphism $f^X\colon(E_h)_*\simto (E_h)_*(X)$ and let $\iota_X=f^X(1)\in (E_h)_0(X)$. The differential \[d_{2p-1}^{X}\colon E^{0,0}_{2p-1}(X)\longrightarrow E^{2p-1,2p-2}_{2p-1}(X)\] is determined by the image of $\iota_X$. Define a homomorphism $\phi^X$ via the following commutative diagram:
		\begin{equation*}
			\begin{tikzcd}
				H_c^0(\G_h;\pi_0(E_h))\rar[dashed,"\phi^X"]\dar["(f^X)_*"',"\cong"]&H_c^{2p-1}(\G_h;\pi_{2p-2}(E_h))\dar["(f^X)_*","\cong "']\\
				H_c^0(\G_h;(E_h)_0(X))\rar["d_{2p-1}^X"]&H_c^{2p-1}(\G_h;(E_h)_{2p-2}(X))
			\end{tikzcd}
		\end{equation*}
		One can check that $\phi^X(1)$ is independent of the choice of $f^X$. We define the next detection map $\ev_2\colon\kappa_h\to H_c^{2p-1}(\G_h;\pi_{2p-2}(E_h))$ by setting $\ev_2(X):=\phi^X(1)$.
	\end{construction}
	\begin{prop}
		The map $\ev_2\colon\kappa_h\to H_c^{2p-1}(\G_h; \pi_{2p-2}(E_h))$ is a group homomorphism.
	\end{prop}
	\begin{proof}
		It suffices to check $\ev_2(X\wedge_{K(h)} Y)=\ev_2(X)+\ev_2(Y)$. This follows from the K\"{u}nneth isomorphism which is compatible with the $\G_h$-actions:
		\[(E_h)_*(X\wedge_{K(h)}Y)\cong (E_h)_*X\otimes_{(E_h)_*}(E_h)_*Y.\]
		This implies 
		\begin{align*}
			E_{2p-1}^{s,t}(X\wedge_{K(h)} Y)&=E_{2}^{s,t}(X\wedge_{K(h)} Y)\\&\cong E_{2}^{s,t}(X)\otimes_{E_2^{0,0}(S^0)}E_{2}^{s,t}(Y)\\&=E_{2p-1}^{s,t}(X)\otimes_{E_{2p-1}^{0,0}(S^0)}E_{2p-1}^{s,t}(Y).
		\end{align*}
		Now by the multiplicative structure of the spectral sequence and the Leibniz rule, we have
		\begin{align*}
			&d^{X\wedge_{K(h)} Y}_{2p-1}(\iota_X\wedge  \iota_Y)=d_{2p-1}^X(\iota_X)\otimes \iota_Y+\iota_X\otimes d_{2p-1}^Y(\iota_Y)&&\\\implies& \ev_2(X\wedge_{K(h)}  Y)=\phi^{X\wedge_{K(h)} Y}(1)=\phi^X(1)+\phi^Y(1)=\ev_2(X)+\ev_2(Y).&&\qedhere
		\end{align*}
	\end{proof}
	\begin{prop}\label{prop:ev_inj}
		The map $\ev_2\colon\kappa_h\to H^{2p-1}_c(\G_h;\pi_{2p-2}(E_h))$ is injective when $4p-3>h^2$ and $(p-1)\nmid h$. In particular, it is injective when $2p-1=h^2$.
	\end{prop}
	\begin{proof}
		For any $X\in \ker\ev_2$, a unit $[\iota_X]$ in $E_2^{0,0}(X)$ does not support a $d_{2p-1}$-differential. By Sparseness (\Cref{lem:sparseness}), the next possible non-trivial differential is $d_{4p-3}^X\colon E_{4p-3}^{0,0}(X)\to E_{4p-3}^{4p-3,4p-2}(X)$. The target of this differential is zero, since it is above the horizontal vanishing line at $s=h^2$ under our assumption. The same argument shows $[\iota_X]$ does not support any higher differentials and is thus a permanent cycle. The rest of the proof is identical to that of \Cref{thm:h2<2p-1}.
	\end{proof}
	This finishes the first implication II$\implies$I in \Cref{subsec:strategy}. The goal of this paper is to answer the following question:
	\begin{question}\label{quest:kappa_h}
		Is $\kappa_h=0$ when $2p-1=h^2$? 
	\end{question}	
	\Cref{prop:ev_inj} implies this would be true if 
	\[H_c^{2p-1}(\G_h;\pi_{2p-2}(E_h))=H_c^{h^2}(\G_h;\pi_{2p-2}(E_h))=0. \]
	\subsection{A filtration on $K(h)$-local Picard groups}\label{subsec:Pic_filtration}
	The main results of this paper do not depend on this subsection. Following the construction above, one can define $\kappa_h^{(1)}:=\ker \ev_2$ and construct the next algebraic detection map using the $d_{4p-3}$-differential:
	\begin{equation*}
		\ev_3\colon \kappa_h^{(1)}\longrightarrow E_{2p}^{4p-3,4p-4}(S^0)=E_{4p-3}^{4p-3,4p-4}(S^0).
	\end{equation*}
	Eventually, we get a descent filtration on $\Pic_{K(h)}$ (see \cite[\S3.3]{BBGHPS_exotic_h2_p2}):
	\begin{equation}\label{eqn:tower}
		\begin{tikzcd}[column sep=10ex,
			/tikz/column 1/.append style={anchor=center}
			,/tikz/column 2/.append style={anchor= west},row sep=tiny]
			\cdots\dar[symbol=\subseteq]&\cdots\\
			\kappa_h^{(m)}\rar["\ev_{m+2}"]\dar[symbol=\subseteq]&E_{2m(p-1)+2}^{2(m+1)(p-1)+1,2(m+1)(p-1)}\\
			\cdots\dar[symbol=\subseteq]&\cdots\\
			\kappa_h^{(1)}\rar["\ev_3"]\dar[symbol=\subseteq]&E_{2p}^{4p-3,4p-4}\\
			\kappa_h\rar["\ev_2"]\dar[symbol=\subseteq]&E_2^{2p-1,2p-2}=H_c^{2p-1}(\G_h;\pi_{2p-2}(E_h))\\
			\Pic^0_{K(h)}\rar["\ev_1"]\dar[symbol=\subseteq]&\Pic(\G_h\text{-}\pi_0(E_h)\text{-modules})\cong H^1_c(\G_h;\pi_0(E_h)^\x)\\
			\Pic_{K(h)}\rar[twoheadrightarrow,"\ev_0"]&\Pic(\text{graded }{(E_h)_*}\text{-modules})\cong \Z/2.
		\end{tikzcd}
	\end{equation}
	Each term in this tower is the kernel of the horizontal detection map right below it.  
	\begin{rmk}\label{rem:vanishing_line}
		For each fixed $p$ and $h$, \eqref{eqn:tower} is a finite (hence Hausdorff) filtration on $\kappa_h$. This is because the HFPSS \eqref{eqn:HFPSS_X} for $S^0_{K(h)}$ has a horizontal vanishing line on the $E_r$-page when $r$ is large enough by \cite[Theorem 2.3.9]{Beaudry-Goerss-Henn}. As a result, the target of $\ev_m$ will eventually be zero and $\kappa_h^{(m)}=\kappa_h^{(m+1)}=\cdots=0$ when $m\gg 0$.
	\end{rmk}
	The right column in \eqref{eqn:tower} is the $0$-stem of a  spectral sequence (similar to the one found in \cite[Theorem 3.2.1]{MS_Picard}) to compute the homotopy groups of the Picard \emph{spectrum} $\mathfrak{pic}_{K(h)}$ for $\Sp_{K(h)}$. Indeed, $\pi_0\left(\mathfrak{pic}_{K(h)}\right)=\Pic_{K(h)}$. In a recent paper \cite{Heard_2021Sp_kn-local}, Heard has proved the following:
	\begin{thm}[{\cite[Example 6.18]{Heard_2021Sp_kn-local}}]\label{thm:DSS_pic}
		There is a descent spectral sequence (DSS) for $\mathfrak{pic}_{K(h)}$ that converges when $t-s\ge 0$, whose $E_2$-page is:
		\begin{equation*}
			E_2^{s,t}=\left\{\begin{array}{ll}
				0,&t<0;\\
				\Z/2,&s=t=0;\\
				H^s_c(\G_h;\pi_0(E_h)^\x),&t=1;\\
				H^s_c(\G_h;\pi_{t-1}(E_h)),&t\ge 2,
			\end{array}\right. \Longrightarrow \pi_{t-s}\left(\mathfrak{pic}_{K(h)}\right).
		\end{equation*}
	\end{thm}
	Let's analyze the $-1,0,1$-columns on the $E_2$-page of the descent spectral sequence \Cref{thm:DSS_pic}, illustrated below in Adams grading.	
	\begin{figure}[ht]
		\begin{tikzpicture}[xscale=4]
			\node at (-1,0) {$0$};
			\node at (0,0) {$\Z/2$};
			\node at (1,0) {$\Zpx$};
			\node at (-1,1) {$?$};
			\node at (0,1) (A) {$H^1_c(\G_h;\pi_0(E_h)^\x)$};
			\node at (1,1) {$0$};
			\node at (-1,2) {$H^2_c(\G_h;\pi_0(E_h)^\x)$};
			\node at (0,2) {$0$};
			\node at (1,2) {$\cdots$};
			\node at (-1,3) {$0$}; 
			\node at (0,3) {$\cdots$};
			\node at (1,3) {$0$};
			\node at (-1,4) {$\cdots$};
			\node at (0,4) {$0$};
			\node at (-1,5) {$0$};
			\node at (1,4) (E) {$H^{2p-2}_c(\G_h,\pi_{2p-2}(E_h))$};
			\node at (0,5) (C) {$H^{2p-1}_c(\G_h,\pi_{2p-2}(E_h))$};
			\node at (1,5) {$0$};
			\node at (-1,6) (B) {$H^{2p}_c(\G_h,\pi_{2p-2}(E_h))$};
			\node at (0,6) {$0$};
			\draw[color=red,thick,dashed, ->] (A) -- node[below left,color=red] {$d_{2p-1}$?} (B);
			\node at (1,6) {$\cdots$};
			\node at (-1,7) {$0$}; 
			\node at (0,7) {$\cdots$};
			\node at (1,7) {$0$};
			\node at (-1,8) {$\cdots$};
			\node at (0,8) {$0$};
			\node at (1,8) {$H^{4p-4}_c(\G_h,\pi_{4p-4}(E_h))$};
			\node at (-1,9) {$0$};
			\node at (0,9) (F) {$H^{4p-3}_c(\G_h,\pi_{4p-4}(E_h))$};
			\node at (1,9) {$0$};
			\node at (-1,10) (D) {$H^{4p-2}_c(\G_h,\pi_{4p-4}(E_h))$};
			\node at (0,10) {$0$};
			\node at (1,10) {$\cdots$};
			\draw[color=red,thick,dashed, ->] (C) -- node[above right,color=red] {$d_{2p-1}$?} (D);
			\draw[color=red,thick,dashed, ->] (E) -- node[above right,color=red] {$d_{2p-1}$?} (F);
			\draw[color=purple,thick,dashed, ->] (A) -- node[above right,very near start, color=purple] {$d_{4p-3}$?} (D);
			\draw[->] (-1.5,-0.5) -- (1.5,-0.5);
			\draw[->] (-1.5,-0.5) -- (-1.5,10.5);
			\node[below] at (-1.2,-0.5) {$t-s=-1$};
			\node[below] at (0,-0.5) {$0$};
			\node[below] at (1,-0.5) {$1$};
			\node[left] at (-1.5,0) {$s=0$};
			\node[left] at (-1.5,1) {$1$};
			\node[left] at (-1.5,2) {$2$};
			\node[left] at (-1.5,3) {$3$};
			\node[left] at (-1.5,4) {$\cdots$};
			\node[left] at (-1.5,5) {$2p-1$};
			\node[left] at (-1.5,6) {$2p$};
			\node[left] at (-1.5,7) {$2p+1$};
			\node[left] at (-1.5,8) {$\cdots$};
			\node[left] at (-1.5,9) {$4p-3$};
			\node[left] at (-1.5,10) {$4p-2$};
		\end{tikzpicture}
	\end{figure}
	On this page of the spectral sequence:
	\begin{itemize}
		\item $E_2^{0,0}=H^0_c(\G_h;\Z/2)=\Z/2$. The non-zero element is a permanent cycle, since it represents $S^1$ in $\Pic_{K(h)}$. So $E_\infty^{0,0}=E^{0,0}_2=\Z/2$.
		\item $E_2^{0,1}=H^0_c(\G_h;\pi_0(E_h)^\x)=\Zpx$. This term does not support any higher differential, because they represent permanent cycles $\Zpx\subseteq \pi_0\left(S_{K(h)}^0\right)^\x\cong \pi_1\left(\mathfrak{pic}_{K(h)}\right)$.
		\item $E_2^{1,1}=H^1_c(\G_h;\pi_0(E_h)^\x)=\Pic_{K(h)}^{alg,0}$. For degree reasons, this term cannot be hit by a differential. But it may support one. As a result, $E_\infty^{1,1}$ is a subgroup of $H^1_c(\G_h;\pi_0(E_h)^\x)$.
		\item By \Cref{lem:sparseness}, the next possibly nonzero terms in the $-1,0,1$-stems are when $t=2p-1$. In the $0$-stem, it is $E_2^{2p-1,2p-1}=H^{2p-1}_c(\G_h;\pi_{2p-2}(E_h))$. The only possible differential that could hit this term is $d_{2p-1}\colon E_2^{0,1}\to E_2^{2p-1,2p-1}$. But since elements in $E_2^{0,1}=\Zpx$ are all permanent cycles, this differential is zero. On the other hand, there is room for $E_2^{2p-1,2p-1}$ to support a differential. As a result, $E_\infty^{2p-1,2p-1}$ is a subgroup of $E_2^{2p-1,2p-1}=H^{2p-1}_c(\G_h;\pi_{2p-2}(E_h))$.
	\end{itemize}
	Now we can compare the $E_\infty$-page of the descent spectral spectral sequence for Picard spaces in \Cref{thm:DSS_pic} and the filtration in \eqref{eqn:tower}. Notice when $t\ge 2$, the $E_2^{s,t}$-term in  \Cref{thm:DSS_pic} is the the same as $E_2^{s,t-1}$ in HFPSS \eqref{eqn:HFPSS_X} for $X=S^0_{K(h)}$. The Picard group $\Pic_{K(h)}=\pi_0\left(\mathfrak{pic}_{K(h)}\right)$ is an extension of the terms $E_{\infty}^{s,s}$ in \Cref{thm:DSS_pic}. More precisely, we have a descending filtration $\Pic_{K(h)}=F^0\supseteq F^1\supseteq F^2\supseteq F^3\supseteq\cdots$, where the layers are related by short exact sequences:
	\begin{equation*}
		\begin{tikzcd}
			0\rar&F^{s+1}\rar&F^s\rar&E_\infty^{s,s}\rar&0
		\end{tikzcd},\qquad s\ge 0.
	\end{equation*}
	As is mentioned in \Cref{rem:vanishing_line}, this is essentially a finite filtration since $E_{\infty}^{s,s}=0$ when $s\gg 0$. In this filtration, we have $F^1=\Pic_{K(h)}^0$ and $F^2=F^3=\cdots=F^{2p-1}=\kappa_h$ is the exotic $K(h)$-local Picard group. The $\ev$-maps can then be defined as composite maps:
	\begin{align*}
		&\begin{tikzcd}[ampersand replacement=\&]
			\&E_\infty^{0,0}\dar[equal]\\F^0=\Pic_{K(h)}\ar[ur,->>]\rar["\ev_0"]\& E_2^{0,0}
		\end{tikzcd}&&\begin{tikzcd}[ampersand replacement=\&]
			\&E_\infty^{1,1}\dar[>->]\\F^1=\Pic^0_{K(h)}\ar[ur,->>]\rar["\ev_1"]\& E_2^{1,1}
		\end{tikzcd}\\ 
		&\begin{tikzcd}[ampersand replacement=\&]
			\&E_\infty^{2p-1,2p-1}\dar[>->]\\F^{2p-1}=\kappa_h\ar[ur,->>, end anchor=south west]\rar["\ev_2"]\& E_2^{2p-1,2p-1}
		\end{tikzcd}&&\begin{tikzcd}[ampersand replacement=\&]
			\&E_\infty^{4p-3,4p-3}\dar[>->]\\F^{4p-3}=\kappa_h^{(1)}\ar[ur,->>,end anchor=south west]\rar["\ev_3"]\& E_{2p}^{4p-3,4p-3}
		\end{tikzcd}
	\end{align*}
	For $\ev_3$, the only differential that can hit $E_2^{4p-3,4p-3}$ is $d_{2p-1}$. So $E_{2p}^{4p-3,4p-3}$ cannot be hit by a differential, but it may support one. As a result, $E_\infty^{4p-3,4p-3}$ is a subgroup of $E_{2p}^{4p-3,4p-3}$. 
	
	From the factorizations above, we can see $\ev_1$ and $\ev_2$ are surjective precisely when $E_2^{1,1}=E_\infty^{1,1}$ and $E_2^{2p-1,2p-1}=E_\infty^{2p-1,2p-1}$. This will be the case if the targets of the potential differentials supported at $E_2^{1,1}$ and $E_2^{2p-1,2p-1}$ are above the horizontal vanishing line on the $E_2$-page. 
	\begin{prop}\label{prop:ev_surj} Suppose $(p-1)\nmid h$. \Cref{thm:DSS_pic} implies:
		\begin{enumerate}
			\item \cite[Remark 2.6]{Piotr_2018} The map $\ev_1\colon\Pic^0_{K(h)}\to \Pic^{alg,0}_{K(h)}:=H^1_c(\G_h;\pi_0(E_h)^\x)$ is an isomorphism when $2p-1>h^2$ and is a  surjection when $2p-1=h^2$.
			\item The map $\ev_2\colon \kappa_h\to H^{2p-1}_c(\G_h;\pi_{2p-2}(E_h))$ is an isomorphism when $4p-3>h^2$ and is a surjection when $4p-3=h^2$.  
		\end{enumerate}
	\end{prop}
	\begin{proof}
		The injectivity parts are from \Cref{thm:h2<2p-1} and \Cref{prop:ev_inj}, respectively.
		
		By sparseness (\Cref{lem:sparseness}), the first possible non-trivial differentials supported at the two terms are
		\begin{align*}
			d_{2p-1}\colon & E_2^{1,1}\longrightarrow E_2^{2p,2p-1}=H^{2p}_c(\G_h;\pi_{2p-2}(E_h)),\\
			d_{2p-1}\colon& E_2^{2p-1,2p-1}\longrightarrow E_2^{4p-2,4p-3}=H^{4p-2}_c(\G_h;\pi_{4p-4}(E_h)).
		\end{align*}
		Under the assumptions, the targets of the two $d_{2p-1}$-differentials are above the horizontal vanishing line at $s=h^2$ in the respective cases. As a result, their targets vanish and $E_2^{1,1}=E_\infty^{1,1}$, $E_2^{2p-1,2p-1}=E_\infty^{2p-1,2p-1}$. This proves the surjectivity part.
	\end{proof}
	\begin{rem}\label{rem:ev_surj}
		While the proof of \Cref{prop:ev_surj} depends on \Cref{thm:DSS_pic}, the statements have been verified independent of the descent spectral sequence in many cases, sometimes even without the assumption that $(p-1)\nmid h$:
		\begin{enumerate}
			\item The map $\ev_1$ is known to be surjective when 
			\begin{itemize}
				\item $h=1$ \cite[Corollary 2.6 for $p>2$, Lemma 3.4 for $p=2$]{HMS_picard}.
				\item $h=2,p>2$ \cite[Theorem 2.9]{GHMR_Picard}. 
				\item $2(p-1)>h^2+h$ for general $h$ and $p$ \cite[Theorem 2.5]{Piotr_2018}.
			\end{itemize}
			It is an open question whether the map $\ev_1$ is surjective or not in the $h=p=2$ case. 
			\item The map $\ev_2$ is known to be an isomorphism when 
			\begin{itemize}
				\item $h=1,p=2$ \cite[Remark 3.3]{GHMR_Picard}.
				\item $h=2,p=3$ \cite[Theorem 3.4]{GHMR_Picard}.
			\end{itemize}
		\end{enumerate}
	\end{rem}
	\begin{rem}
		The filtration \eqref{eqn:tower} for $\kappa_2$ at prime $2$ has been completed studied in \cite{BBGHPS_exotic_h2_p2}. In particular, they showed that the detection maps
		\begin{align*}
			\ev_3\colon \kappa_2^{(1)}\to E_4^{5,4} & \text{ is \emph{not} surjective};\\
			\ev_4\colon \kappa_2^{(2)}\to E_6^{7,6} & \text{ is injective}.
		\end{align*}
		See \cite[Theorem 12.30]{BBGHPS_exotic_h2_p2} for the full details. 
	\end{rem}
	We conclude this subsection by noting \Cref{thm:DSS_pic} implies the following:
	\begin{cor}
		When $2p-1=h^2$,  then the followings are equivalent:
		\begin{enumerate}
			\item $\ev_1\colon \Pic_{K(h)}\simto \Pic^{alg}_{K(h)}$ is an isomorphism.
			\item $\ev_1\colon \Pic^0_{K(h)}\simto \Pic^{alg,0}_{K(h)}$ is an isomorphism.
			\item $\kappa_h:=\ker\ev_1=0$.
			\item $H_c^{2p-1}(\G_h;\pi_{2p-2}(E_h))=H_c^{h^2}(\G_h;\pi_{2p-2}(E_h))=0$.
		\end{enumerate}
	\end{cor}
	\begin{proof}
		(1)$\iff$(2) follows from \Cref{cor:ev1_ker}. By \Cref{prop:ev_surj}, $\ev_1$ is surjective and $\ev_2$ is an isomorphism when $2p-1=h^2$. This implies (2)$\iff$(3) and (3)$\iff$(4), respectively.
	\end{proof}
	\section{Duality}
	In \Cref{prop:ev_inj}, we have established that there is an isomorphism
	\[
	\ev_2\colon \kappa_h\simto H^{2p-1}_c(\G_h; \pi_{2p-2}(E_h))
	\]
	under the conditions that $4p-3>h^2$ and $h$ is not divisible by $p-1$. In particular, this is true when $2p-1=h^2$. In light of this injection, we are thus interested in determining the group $H_c^{h^2}(\G_h; \pi_{2p-2}(E_h))$. The purpose of this section is reduce this computation using duality argument. We will prove the successive implications II$\impliedby$III$\impliedby$IV$\impliedby$V mentioned in \Cref{subsec:strategy}:
	\begin{prop}\label{prop:duality_reduction} Suppose $(p-1)\nmid h$.
		\begin{enumerate}
			\item (\Cref{prop:Hh2_mod_p}) $H_c^{h^2}(\G_h;\pi_{2p-2}(E_h))\cong H_c^{h^2}(\G_h;\pi_{2p-2}(E_h)/p)$.
			\item (\Cref{prop:Hh2_reduction}) For a general $t\in \Z$, we have
			\begin{equation*}
				H_c^{h^2}(\G_h;\pi_t(E_h)/p)\cong \left[ \colim_{p\in J\trianglelefteq \pi_0(E_h)}H^0_c\left(\G_h;\left.\pi_{2h-t-\frac{p^N|v_h|}{p-1}}(E_h)\right/ J\right)\right]^\vee,
			\end{equation*}
			where $J\trianglelefteq \pi_0(E_h)$ ranges through all open invariant ideals containing $p$ and $N$ is the smallest integer such that $v_h^{p^N}$ is invariant mod $J$. The colimit system is described in \Cref{defn:mod_m_infty}. 
		\end{enumerate}
	\end{prop}
	\subsection{Reduction to mod-$p$ coefficients}
	The purpose of this subsection is to prove (1) in \Cref{prop:duality_reduction}. This is the second step III$\implies$II in \Cref{subsec:strategy}.
	\begin{lem}[Bounded torsion, \mbox{\cite[page 8]{Goerss-Hopkins}}]\label{lem:bounded_torsion}
		The  cohomology group $H^*_c(\G_h;\pi_{2p-2}(E_h))$ is $p$-torsion.
	\end{lem}
	\begin{prop}\label{prop:Hh2_mod_p}
		If $(p-1)\nmid h$, then we have an isomorphism: 
		\begin{equation*}
			H_c^{h^2}(\G_h;\pi_{2p-2}(E_h))\simto H_c^{h^2}(\G_h;\pi_{2p-2}(E_h)/p).
		\end{equation*}
	\end{prop}
	\begin{proof}
		Let $M=\pi_{2p-2}(E_h)$. There is a short exact sequence of $\G_h$-$\pi_0(E_h)$-modules
		\begin{equation}\label{eqn:mod_p_ses}
			\begin{tikzcd}
				0\arrow[r] & M\arrow[r,"p"] & M\arrow[r] & M/p\arrow[r] & 0.
			\end{tikzcd}
		\end{equation}
		This short exact sequence induces a long exact sequence in cohomology 
		\begin{equation}\label{eqn:mod_p_les}
			\cdots \to H^{k}_c(\G_h; M)\xrightarrow{p}  H^{k}_c(\G_h;M)\to H^k(\G_h; M/p)\xrightarrow{\delta} H^{k+1}(\G_h; M)\to \cdots 
		\end{equation}
		By \Cref{lem:bounded_torsion}, all the multiplication-by-$p$ maps in \eqref{eqn:mod_p_les} are zero.  Since $p-1$ does not divide $h$, $\cd_p(\G)=h^2$ by \Cref{lem:cd}. As a result, the cohomology groups $H^s_c(\G_h;-)=0$ when $s>h^2$. This means the long exact sequence \eqref{eqn:mod_p_les} ends with 
		\[
		0\to H_c^{h^2}(\G_h;M)\to  H_c^{h^2}(\G_h;M/p)\to 0
		\]
		and we get the desired isomorphism.
	\end{proof}
	\begin{rem}
		Let $M=\pi_{2p-2}(E_h)$ as above. When $s=0$, we have $\delta\colon H^0_c(\G_h;M/p)\simto H_c^1(\G_h;M)$. When $1\le s\le h^2-1$, there is a short exact sequence instead:
		\begin{equation*}			
			0\to H_c^s(\G_h;M)\to H_c^s(\G_h;M/p)\xrightarrow{\delta }H_c^{s+1}(\G_h;M)\to 0.
		\end{equation*}
		Since all three groups above are $\F_p$-vector spaces, the short exact sequence splits (non-canonically). As a result, we have $H_c^s(\G_h;M/p)\cong H_c^s(\G_h;M)\oplus H_c^{s+1}(\G_h;M)$ for $1\le s\le h^2-1$.
	\end{rem}
	\begin{rem}
		The claims above hold for any $M=\pi_t(E_h)$, where $t=2m(p-1)$ and $p\nmid m$.
	\end{rem}
	\subsection{Poincar\'e duality}
	The Morava stabilizer group $\G_h$ is not just a profinite group, but is also a compact $p$-adic Lie group of dimension $h^2$. This imposes a great deal of more structures on its (co-)homology. In this section, we review the theory of Poincar\'e duality for $p$-adic analytic groups following \cite{SW}.  Recall that for a property $P$, a profinite group $G$ is said to be virtually $P$ if there is an open normal subgroup of $G$ which is $P$. A profinite group $G$ has Poincar\'e duality of dimension $d$ if 
	
	\[
	H^d_c(G, \Z_p\llb G\rrb )\cong \Z_p
	\]
	as abelian groups (\cite[(4.4.1)]{SW}). 
	
	\begin{thm}[Lazard, {\cite[Theorem 5.1.9]{SW}}]\label{thm:lazard}
		Let $G$ be a compact $p$-adic analytic group. Then $G$ is a virtual Poincar\'e duality group of dimension $d = \dim G$. 
	\end{thm}
	
	In the case of the Morava stabilizer group, $\Sbf_h$ is a virtual Poincar\'e duality group of dimension $h^2$. When $(p-1)\nmid h$, then $\Sbf_h$ contains no $p$-torsion subgroups. In fact, its maximal finite subgroup is cyclic of order $p^h-1$ \cite[Table 5.3.1]{Chromatic_structure}.  Under this assumption, $\Sbf_h$ is a Poincar\'e duality group of dimension $h^2$ (as opposed to a \emph{virtual} one). 
	
	Now $G$ being a profinite group having Poincar\'e duality of dimension $n$ implies that there is a \emph{dualizing module} $D(G)$ such that there are natural isomorphisms \cite[Theorem 4.4.3]{SW} for continuous $G$-modules $M$ that are inverse limits of discrete $G$-modules:
	\[
	H^{n-k}_c(G; M)\longrightarrow H_k^c(G; D(G)\widehat{\otimes}_{\Z_p}M),
	\]
	and for discrete $p$-torsion $G$-modules
	\[
	H_{n-k}^c(G; M)\longrightarrow H^k_c(G; \hom_{\Z_p}(D_p(G), M)).
	\]
	The dualizing module $D(G)$ is given by 
	\[
	D(G) = H^n_c(G; \Zp \llb G \rrb).
	\]
	Note that, as the coefficients $\Zp\llb G\rrb$ has a left $G$-action, the dualizing module $D(G)$ has a corresponding right $G$-action. See \cite[\S 4.5]{BGHS_dualizing_spheres} for further details.
	
	In the case when $G$ is the Morava Stabilizer group $\G_h$, Strickland has calculated the dualizing module $D(\G_h)$ along with its $\G_h$-action.
	
	\begin{thm}[Strickland, \cite{STRICKLAND20001021}]
		As a $\G_h$-module, $H_c^{h^2}(\G_h; \Z_p\llb \G_h\rrb )\cong \Z_p$ has the trivial $\G_h$-action. 
	\end{thm}
	
	\begin{cor}\label{cor:PD}
		Assume $(p-1)\nmid h$. The dualizing module $I_p$ for $\G_h$ is $\Zp^\vee \cong \Q_p/\Z_p$ with the trivial $\G_h$-action. Hence, we have a duality
		\[
		H^{h^2-k}_c(\G_h;M)\cong H^c_k(\G_h; M)
		\]
		that is natural in $p$-profinite continuous $\G_h$-modules $M$.
	\end{cor}	
	\begin{defn}\label{defn:Pontryagin_dual}
		Write $(-)^\vee$ for $\hom_{c}(M, I_p(G))$. If $M$ has a continuous $G$-action, we endow $M^\vee$ with a left $G$-action via 
		\[
		(g\cdot f)(x) := f(g^{-1}x).
		\]
		In the case of $G=\G_h$, \Cref{cor:PD} implies $M^\vee$ is the continuous Pontryagin dual $M^\vee\cong \hom_{c}(M,\Z/p^\infty)$. 
	\end{defn}
	As usual, this also induces a version of Poincar\'e duality for $p$-profinite $\G_h$-modules $M$ in purely cohomological terms when $(p-1)\nmid h$: (\cite[Theorem 4.26]{BGHS_dualizing_spheres})
	\begin{equation}\label{eqn:coh_PD}
		H^k_c(\G_h; M)\cong H^{h^2-k}_c(\G_h;M^\vee)^\vee.
	\end{equation}
	
	\begin{cor}\label{cor:PD_Et}
		Assume $(p-1)\nmid h$. We have the following duality:
		\begin{align}
			H^{h^2}_c(\Sbf_h;\pi_t(E_h))&\cong H_0(\Sbf_h;\pi_t(E_h)),& H^{h^2}_c(\Sbf_h;\pi_t(E_h)/p)&\cong H_0(\Sbf_h;\pi_t(E_h)/p);\label{eqn:PD_homology}\\ 
			H^{h^2}_c(\Sbf_h;\pi_t(E_h))&\cong H^0_c(\Sbf_h;\pi_t(E_h)^\vee)^\vee,& H^{h^2}_c(\Sbf_h;\pi_t(E_h)/p)&\cong H^0_c(\Sbf_h;(\pi_t(E_h)/p)^\vee)^\vee.\label{eqn:PD_cohomology}
		\end{align}
	\end{cor}
	\begin{rem}
		Using the duality \eqref{eqn:PD_homology}, we can give another proof of \Cref{prop:Hh2_mod_p} by showing:
		\begin{enumerate}
			\item The group homology $H_*(\G_h;\pi_{2p-2}(E_h))$ is $p$-torsion. This is because the orbit of the action by $\Zpx\subseteq \Sbf_h$ is already $p$-torsion.
			\item Apply $H_*$ to the short exact sequence \eqref{eqn:mod_p_ses} to get the a long exact sequence like \eqref{eqn:mod_p_les}. Equivalently, we are essentially applying \eqref{eqn:PD_homology} to every term in \eqref{eqn:mod_p_les}.
		\end{enumerate}		
	\end{rem}
	\subsection{Gross-Hopkins duality}\label{subsec:GH_duality}
	Now we want to use \eqref{eqn:PD_cohomology} to compute $H^{h^2}_c(\G_h;M/p)$ where $M=E_{t}$. To do so, we have to identify the $\G_h$-equivariant Pontryagin dual of $M$. This is realized by Gross-Hopkins duality.
	\begin{rem}
		For the purpose of \Cref{quest:kappa_h}, we only need to study the case when $t=2p-2$. Later for the Vanishing Conjecture, we also need the $t=0$ case. So we will give a uniform treatment for all $t\in \Z$ in the remainder of this section.  
	\end{rem}
	
	We remind the reader the definition of the determinant twist. The group $\Sbf_h$ can be realized as a subgroup of $\mathrm{GL}_{h}(\W)$. Thus, taking the determinant, we have a map 
	\[
	\det\colon \Sbf_h\to \W^\times.
	\]
	It turns out that this map actually factors through $\Z_p^\times$. We extend this to the extended Morava stabilizer group via the composite
	\[
	\begin{tikzcd}
		\det\colon \G_h\cong \Sbf_h\rtimes \Gal\arrow[r]&  \Z_p^\times \times \Gal \arrow[r,"proj"] & \Z_p^\times. 
	\end{tikzcd}
	\]
	This results in a $\G_h$-action on $\Z_p$. 	
	\begin{defn}\label{defn:det}
		The $\G_h$-action on $\Z_p$ above is denoted by $\Z_p\langle \det\rangle$. Given a Morava module $M$ we write $M\langle \det\rangle$ for the Morava module
		\[
		M\langle \det \rangle \cong M\otimes_{\Z_p}\Z_p\langle \det\rangle
		\]
		with the diagonal $\G_h$-action. We refer to $M\langle \det\rangle$ as the \emph{determinant twist} of $M$.
	\end{defn}
	\begin{defn}\label{defn:mod_m_infty}
		We now describe the quotient mod $\mfrak^{\infty}$. Let $M$ be a $\G_h$-$\pi_0(E_h)$-module, we define
		\begin{equation}\label{eqn:mod_m_infty}
			M/\mfrak^\infty:=\colim_{J\trianglelefteq \pi_0(E_h)}M/J,
		\end{equation}
		where $J$ ranges over all open invariant ideals of $\pi_0(E_h)$.  Suppose $J\subseteq J'$ is an inclusion of open invariant ideals of $\pi_0(E_h)$. Then we have a $\G_h$-equivariant isomorphism:
		\begin{equation*}
			M/J'\cong \{[m]\in M/J\mid x\cdot [m]=0,\forall x\in J'\}.
		\end{equation*} 
		This gives the structure map $M/J'\to M/J$ in the colimit system. Similarly, in the mod-$p$ case, we have
		\[M/(p,u_1^\infty,\cdots,u_{h-1}^\infty):=\colim_{p\in J\trianglelefteq \pi_0(E_h)}M/J,  \]
		where $J$ ranges over all invariant ideals of $\pi_0(E_h)$ containing $p$.
	\end{defn}
	\begin{thm}[Gross-Hopkins]\label{thm:GH_dual} Let $\mfrak\trianglelefteq \pi_0(E_h)$ be the maximal ideal.
		\begin{enumerate}
			\item \textup{\cite{STRICKLAND20001021}} There is a $\G_h$-equivariant perfect pairing of $\G_h$-$\pi_0(E_h)$-modules: 
			\[ \rho\colon\pi_0(E_h)/\mfrak^\infty\otimes_{\pi_0(E_h)}\Omega^{h-1}\longrightarrow \Qp/\Zp,\]
			where $\Omega^{h-1}$ is the top exterior power of the module of continuous K\"ahler differentials for $\pi_0(E_h)$ relative to $\W$. 
			\item \textup{\cite{GrossHopkins}} The module $\Omega^{h-1}$ is $\G_h$-equivariantly equivalent to the bundle $\omega^{\otimes h}\ldetr$ over the Lubin-Tate deformation space, where $\omega=\pi_2(E_h)$ is the sheaf of invariant of differentials and $\ldetr$ is the determinant twist. 
		\end{enumerate}
	\end{thm}
	\begin{cor}[See {\cite[Proposition 19]{STRICKLAND20001021}}]\label{cor:GH_dual} 
		The $\G_h$-equivariant Pontryagin dual of $\pi_t(E_h)$ is
		\[\left(\pi_t(E_h)\right)^\vee\cong (\pi_{2h-t}(E_h))\ldetr/\mfrak^{\infty}. \]
	\end{cor}
	\begin{proof}
		The $\G_h$-equivariant perfect pairing $\rho$ in \Cref{thm:GH_dual} can be rewritten as:
		\begin{align*}
			\rho\colon \pi_0(E_h)/\mfrak^\infty\otimes_{\pi_0(E_h)}\Omega^{h-1}\cong \pi_t(E_h)\otimes_{\pi_0(E_h)}\pi_{-t}(E_h)/\mfrak^\infty\otimes_{\pi_0(E_h)}\Omega^{h-1}\longrightarrow \Qp/\Zp.
		\end{align*}
		This implies the $\G_h$-equivariant Pontryagin dual of $\pi_t(E_h)$ is $\pi_{-t}(E_h)/\mfrak^\infty\otimes_{\pi_0(E_h)}\Omega^{h-1}$, which is $\G_h$-equivariantly isomorphic to $(\pi_{2h-t}(E_h))\ldetr/\mfrak^{\infty}$ by  part (2) of \Cref{thm:GH_dual}. 
	\end{proof}
	Applying \eqref{eqn:coh_PD}, we have proved:
	\begin{equation}\label{eqn:GH_duality}
		H^{h^2}_c(\G_h;\pi_t(E_h))\cong H_c^0\left(\G_h;(\pi_{2h-t}(E_h))\ldetr/\mfrak^\infty\right)^\vee.
	\end{equation}
	The formula holds with $\pi_t(E_h)$ replaced by $\pi_t(E_h)/p$. This yields the third implication IV$\implies$III in \Cref{subsec:strategy} when $t=2p-2$.  Notice \eqref{eqn:mod_m_infty} is a filtered colimit, and the group $\G_h$ is topologically finitely generated (since it is a finite dimensional $p$-adic Lie group), we have 
	\begin{prop}
		There are isomorphisms:
		\begin{align*}
			\colim_{J\trianglelefteq E_h} H^0_c(\G_h;M/J)\simto& H^0_c(\G_h;M/\mfrak^{\infty}),\\
			\colim_{p\in J\trianglelefteq E_h} H^0_c(\G_h;M/J)\simto& H^0_c(\G_h;M/(p,u_1^\infty,\cdots,u_{h-1}^\infty)).
		\end{align*}
	\end{prop}
	Now set $M=E_{2h-2p+2}\ldetr$. In order to prove
	\[  H^0_c(\G_h;M/(p,u_1^\infty,\cdots,u_{h-1}^\infty))^\vee =0, \]
	it suffices to show $H^0_c(\G_h;M/J)=0$ for a cofinal system of invariant ideals $J\trianglelefteq \pi_0(E_h)$ containing $p$. To do that, we need to identify the determinant twist $\pi_0(E_h)\ldetr$ mod $p$. The following theorem was originally stated in \cite[Corollary 7]{Hopkins_1994} and a nice proof appears in \cite[Theorem 1.32]{Goerss-Hopkins}:	
	\begin{thm}[Gross-Hopkins]\label{thm:det_mod_p}
		When $p>2$, there is an isomorphism of $\G_h$-$\pi_0(E_h)$-modules: \[\pi_0(E_h)\ldetr/p\cong \pi_0\left(\Sigma^{\lim\limits_{N\to \infty}\frac{p^N|v_h|}{p-1}}E_h\right)/p. \]
		More precisely, let $J\trianglelefteq \pi_0(E_h)$ be an open invariant ideal containing $p$, such that $v_h^{p^N}$ is invariant modulo $J$, then \[\pi_0(E_h)\ldetr/J \cong \pi_0\left(\Sigma^{\frac{{p^N|v_h|}}{p-1}}E_h\right)/J. \]
	\end{thm}
	\begin{rem}
		Suppose $v_h^{p^{N'}}$ is also invariant mod $J$ for some $N'<N$. Then
		\[\pi_0(E_h)\ldetr/J\cong \pi_0\left(\Sigma^{\frac{{p^{N'}|v_h|}}{p-1}}E_h\right)/J.\]
		This is compatible with the statement in \Cref{thm:det_mod_p}. This is because
		\begin{align*}
			&\frac{{p^{N'}|v_h|}}{p-1}\equiv \frac{{p^{N}|v_h|}}{p-1}\mod p^{N'}|v_h|\\\implies &\pi_0\left(\Sigma^{\frac{{p^{N'}|v_h|}}{p-1}}E_h\right)/J\cong \pi_0\left(\Sigma^{\frac{{p^{N}|v_h|}}{p-1}}E_h\right)/J.
		\end{align*}
	\end{rem}
	For each open invariant ideal $J$, there is a smallest $N$ such that $v_h^{p^N}$ is invariant mod $J$. It follows from this proposition that\[M/J=\pi_{2h-2p+2}(E_h)\ldetr/J\cong \left.\pi_{2h-2p+2-\frac{p^N|v_h|}{p-1}}(E_h)\right/J. \]
	Combining all the duality arguments in  \Cref{cor:PD_Et} and \Cref{cor:GH_dual} with the identification of the determinant twist $\pi_0(E_h)\ldetr$ mod $p$ in \Cref{thm:det_mod_p}, we have proved part (2) in \Cref{prop:duality_reduction}.
	\begin{prop}\label{prop:Hh2_reduction} Suppose $(p-1)\nmid h$. Then there is an isomorphism:
		\begin{equation*}
			H_c^{h^2}(\G_h;\pi_t(E_h)/p)\cong \left[\colim_{p\in J\trianglelefteq \pi_0(E_h)}H^0_c\left(\Sbf_h;\left.\pi_{2h-t-\frac{p^N|v_h|}{p-1}}(E_h)\right/ J\right)^\Gal\right]^{\vee},
		\end{equation*}
		where $J\trianglelefteq \pi_0(E_h)$ ranges through all opening invariant ideals containing $p$ and $N$ is the smallest integer such that $v_h^{p^N}$ is invariant mod $J$.
	\end{prop}
	From this, we get the implication V$\implies$IV in \Cref{subsec:strategy}. Consequently, \Cref{quest:kappa_h} now reduces to checking
	\begin{equation}\label{eqn:E_2p-2_duality_reduction}
		H^0_c\left(\G_h;\left.\pi_{2h-2p+2-\frac{p^N|v_h|}{p-1}}(E_h)\right/J\right)=0
	\end{equation} 
	for a cofinal system of invariant ideals $J$ containing $p$, where $N$ is the smallest number such that $v_h^{p^N}$ is invariant mod $J$.
	\subsection{The Chromatic Vanishing Conjecture}
	A closely related computation is the Chromatic Vanishing Conjecture. Consider the natural inclusion $\iota: \W\hookrightarrow \pi_0(E_h)$, which is $\G_h$-equivariant. Explicit computations at height $2$ in \cite{Beaudry_k2_moore,Beaudry-Goerss-Henn,GHMR_Picard,Kohlhaase_IwasawaLT, HKM_K2_Moore,Shimomura-Yabe} show that this inclusion induces isomorphisms in group cohomology of $\G_2$ for all primes and degrees. At $h=p=2$, this isomorphism plays an essential role in disproving and completely understanding the Chromatic Splitting Conjecture by Beaudry-Goerss-Henn in \cite{Beaudry-Goerss-Henn}. Observing this phenomenon, Hans-Werner Henn first raised the question if there is a conceptual reason for the isomorphisms. This leads to a more general conjecture:
	\begin{conjecture}[Chromatic Vanishing Conjecture, {\cite[Conjecture 1.1]{Beaudry_orbits}}, {\cite[Conjecture 1.1.4]{Beaudry-Goerss-Henn}}]\label{conj:CVC} 
		The followings are true for all heights $h$, primes $p$, and (co)-homological degrees $s$:
		\begin{enumerate}
			\item (Integral) The continuous group cohomology and homology of $\coker(\iota)$ vanish so that
			\begin{align*}
				\iota_*\colon H^s_c(\G_h;\W)\simto H^s_c(\G_h;\pi_0(E_h)),&&\iota_*\colon H_s(\G_h;\W)\simto H_s(\G_h;\pi_0(E_h)).
			\end{align*}
			\item (Reduced) The continuous group cohomology and homology of $\coker(\iota\otimes \W/p)$ vanish so that
			\begin{align*}
				\iota_*\colon H^s_c(\G_h;\F_p)\simto H^s_c(\G_h;\pi_0(E_h)/p),&&\iota_*\colon H_s(\G_h;\F_p)\simto H_s(\G_h;\pi_0(E_h)/p).
			\end{align*}
		\end{enumerate}
	\end{conjecture}
	\begin{rem}[{\cite[page 692]{Beaudry_orbits}}]~\label{rem:vanishing_conj}
		\begin{enumerate}
			\item By \Cref{cor:PD} and \eqref{eqn:coh_PD}, the cohomological and homological versions of \Cref{conj:CVC} are equivalent when $(p-1)\nmid h$.
			\item The reduced version of conjecture implies the integral version by the Five Lemma and a $\lim^1$ exact sequence.
			\item The conjecture is a tautology when $h=1$, since $\Zpx$ acts on $\pi_0(E_1)\cong \Zp$ trivially. 
			\item At $h=2$, the conjecture has been proved for all primes.
			\item The proof for $s=0$ at all heights can be found in \cite[Lemma 1.33]{Bobkova-Goerss}.
		\end{enumerate}
	\end{rem}
	\begin{rem}[Hopkins, {\cite[Theorem 8.1]{S_E2}},  {\cite[\S5.3]{Hopkins_AWS2019}},  {\cite{Lader_thesis} }for $p\ge 5$; {Karamanov \cite{Karamanov_Picard}} for $p=3$] When $h=2$ and $p\ge 3$, the \emph{additive} Vanishing Conjecture in cohomological degree $1$ can be used to show a \emph{multiplicative} version of the conjecture:
		\begin{equation*}
			H^1_c(\G_h;\W^\times)\simto H^1_c(\G_h;\pi_0(E_h)^\times).
		\end{equation*}
		From there, we can compute the algebraic $K(2)$-local Picard groups when $p\ge 3$:
		\[\Pic_{K(2)}^{alg,0}\cong \Zp\oplus\Zp\oplus\Z/(p^2-1).\] 
		Combined with \Cref{prop:ev_inj} and \Cref{rem:ev_surj}, we know $\Pic^{alg}_{K(2)}\cong \Pic_{K(2)}\cong\Zp\oplus\Zp \oplus \Z/|v_2|$ when $p\ge 5$. The group is topologically generated by $S^1_{K(2)}$ and $S^0_{K(2)}\ldetr$. Those two generators are related by \Cref{thm:det_mod_p} and the fact that $\ev_1\colon \Pic_{K(2)}\simto \Pic^{alg}_{K(2)}$ is an isomorphism when $p\ge 5$:
		\[S^{0}\ldetr\wedge_{K(2)} V(1)\simeq S^{2(p+1)}\wedge_{K(2)} V(1). \]
	\end{rem}
	
	The case of \Cref{conj:CVC} relevant to \Cref{quest:kappa_h} is if the following holds when $(p-1)\nmid h$:
	\begin{align*}
		&\iota_*\colon \F_p=H_0(\G_h;\F_p)\simto H_0(\G_h;\pi_0(E_h)/p)\\
		\iff&\iota_*\colon \F_p=H^{h^2}_c(\G_h;\F_p)\simto H^{h^2}_c(\G_h;\pi_0(E_h)/p).
	\end{align*}
	As this is the reduced version of \Cref{conj:CVC} in homological degree $0$, we will call it the \textbf{Reduced Homological Vanishing Conjecture} (RHVC). It follows immediately that 
	\begin{equation}\label{eqn:RHVC}
		H_0(\G_h;\pi_0(E_h)/p)\cong H_0(\G_h;\F_{p^h})
		\cong \F_p.\tag{RHVC}
	\end{equation}
	This is the formula we want to prove.  Setting $t=0$ in \Cref{prop:Hh2_reduction}, we get an isomorphism when $(p-1)\nmid h$:
	\[H_c^{h^2}(\G_h;\pi_0(E_h)/p)\cong \left[\colim_{p\in J\trianglelefteq \pi_0(E_h)}H^0_c\left(\G_h;\left.\pi_{2h-\frac{p^N|v_h|}{p-1}}(E_h)\right/ J\right)\right]^\vee.\]	
	As a result, to prove \eqref{eqn:RHVC}, it suffices to show that 
	\begin{equation}\label{eqn:E_0_duality_reduction}
		H^0_c\left(\G_h;\left.\pi_{2h-\frac{p^N|v_h|}{p-1}}(E_h)\right/ J\right)=\F_p
	\end{equation}
	for a cofinal system of invariant ideals $J$ containing $p$, where $N$ is the smallest number such that $v_h^{p^N}$ is invariant mod $J$, and that the structure maps in the colimit are non-zero.
	\section{Greek letter elements}	
	\subsection{The change of rings theorem}
	In this section we will prove the main theorems. The first step is to translate \eqref{eqn:E_2p-2_duality_reduction} and \eqref{eqn:E_0_duality_reduction} to \textbf{Greek letter element} computations in chromatic homotopy theory. We refer readers to \cite[\S1 and \S3]{MRW} and \cite[\S5.1]{green} for an introduction. The transition from $\G_h$-$\pi_0(E_h)$-modules to $BP_*BP$-comodules is achieved by the following theorem:
	\begin{thm}[Morava's Change of Rings Theorem, {\cite[Theorem 6.5]{Devinatz_CoR}}]\label{thm:CoR} Let $M$ be a $BP_*BP$-comodule such that $I_h^n M=0$ for some $n$, where $I_h=(p,u_1,\cdots, u_{h-1})$.  Then there is a natural isomorphism:
		\[r_*\colon \Ext^{s,t}_{BP_*BP}(BP_*,v_h^{-1}M)\simto H^{s}_c(\G_h;\pi_t(E_h)\otimes_{BP_*}M), \]
		where $r_*$ is induced by a ring homomorphism $r\colon BP_*\to \pi_*(E_h)$ defined below:
		\begin{equation*}
			r(v_i)=\left\{\begin{array}{cl}
				u_iu^{1-p^i},&i<h;\\
				u^{1-p^h},&i=h;\\
				0,&i>h.
			\end{array}\right.
		\end{equation*}
	\end{thm}
	Let $p\in J\trianglelefteq \pi_0(E_h)$ be an open invariant ideal containing $p$. For our computation, $M$ is a $BP_*BP$-comodule such that 
	\begin{equation*}
		\pi_0(E_h)\otimes_{BP_*}M\cong \pi_0(E_h)/J.
	\end{equation*}
	\begin{lem}\label{lem:inv_ideal_E_BP}
		When $J=(p,u_1^{j_1},\cdots,u_{h-1}^{j_{h-1}})$, we can take $M:=BP_*/J'$, where $J'=(p,v_1^{j_1},\cdots, v_{h-1}^{j_{h-1}})$. 
	\end{lem}
	The implication VI$\implies$V in \Cref{subsec:strategy} then follows from \Cref{thm:CoR}. We now need to compute $\Ext^{0,t}_{BP_*BP}(BP_*,v_h^{-1}BP_*/J')$ for a family of invariant ideals $J'$ and certain values of $t$. 
	\subsection{Families of Greek letter elements}
	From now on, for a graded $BP_*BP$-comodule $M$, we will write
	\[H^{0,t}(M):=\Ext^{0, t}_{BP_*BP}(BP_*,M).\]
	Suppose $J'=(p,v_1^{j_1},\cdots,v_{h-1}^{j_{h-1}})$ for some $j_i\ge 0$.
	The right hand term can be more explicitly identified as the submodule of primitive elements $x$ of degree $t$ in the comodule $M_1^{h-1}:=v_h^{-1}BP_*/(p, v_1^\infty, \ldots, v_{h-1}^\infty)$, such that $v_i^{j_i}x=0$ for all $1\le i\le h-1$. This establishes the final implication VII$\implies$VI in \Cref{subsec:strategy}.
	
	As a result, we need to compute $H^{0,t}(M_1^{h-1})$. The computation of this $\Ext$-group in general heights are beyond our reach, but we can at least place elements within three distinct families.  
	\begin{prop}\label{prop:families}
		Let $M^m_{h-m}=v_h^{-1}BP_*/(p,v_1,\cdots, v_{h-m-1},v_{h-m}^{\infty},\cdots, v_{h-1}^\infty)$. Then for $0\le m<h$, the cohomology group $H^{0,*}(M^m_{h-m})$ is generated as an $\F_p$-vector space by elements of the following families:
		\begin{enumerate}[label=\textup{\Roman*.}]
			\item $\frac{v^s_h}{pv_1\cdots v_{h-1}}$, where $(s,p)=1$.
			\item $\frac{1}{pv_1^{d_1}\cdots v_{h-1}^{d_{h-1}}}$, where $(p,v_1^{d_1},\cdots,v_{h-1}^{d_{h-1}})$ is an invariant ideal and $d_1=\cdots=d_{h-m-1}=1$.
			\item $\frac{y^s_{m,N}}{pv_1^{d_1}\cdots v_{h-1}^{d_{h-1}}}$, where $(p,v_1^{d_1},\cdots,v_{h-1}^{d_{h-1}}, y^s_{m,N})$ is an invariant ideal with $d_1=\cdots=d_{h-m-1}=1$, $y_{m,N}\equiv y_{m-1,N}\mod (p,v_1,\cdots,v_{h-m})$, $N\ge 1 $ and $(s,p)=1$.
		\end{enumerate}
		Here, the degrees of elements are given by: \[\left|\frac{y^s_{m,N}}{pv_1^{d_1}\cdots v_{h-1}^{d_{h-1}}}\right|=sp^N|v_h|-\sum_{i=1}^{h-1} d_i|v_{i}|.\]
	\end{prop}
	\begin{proof}
		We prove this by induction on $m$.  By \cite[Proposition 5.1.12]{green},  the zeroth cohomology of $M_h^0=v_h^{-1}BP_*/I_{h}$ is $\F_p[v_h^{\pm 1}]$. Identifying the $M_h^0\subseteq M^{h-1}_1$ as a subcomodule consisting of elements that are $v_i$-torsion for all $1\le i\le h-1$, we have proved the $m=0$ case where $y_{0,N}=v_h^{p^N}$.
		
		The $m=1$ case was proved by Miller-Ravenel-Wilson in \cite[Theorem 5.10]{MRW} (see full statements in \Cref{thm:MRW_M11} and \Cref{thm:MRW_M1_h-1}).  Their inductive step from $m=0$ to $m=1$ also applies to the $m>1$ case, as summarized below. Recall that there are short exact sequences of $BP_*BP$-comodules
		\[
		0\to  M^{m}_{h-m}\longrightarrow M^{m+1}_{h-m-1}\xrightarrow{\cdot v_{h-m-1}} M^{m+1}_{h-m-1}\to 0, 
		\]
		which leads to the $v_{h-m-1}$-Bockstein spectral sequence
		\[
		H^{s,t}( M^{m}_{h-m})\otimes \F_p[v_{h-m-1}]/(v_{h-m-1}^\infty)\Longrightarrow H^{s,t}(M^{m+1}_{h-m-1}). 
		\]
		Alternatively, we can consider the long exact sequence of cohomology groups
		\[0\to H^0\left(M^{m}_{h-m}\right)\longrightarrow H^0\left(M^{m+1}_{h-m-1}\right)\xrightarrow{\cdot v_{h-m-1}} H^0\left(M^{m+1}_{h-m-1}\right)\xrightarrow{\delta} H^1\left( M^{m}_{h-m}\right)\to \cdots \]
		As a result, $H^0\left(M^{m}_{h-m}\right)$ is the subgroup of $v_{h-m+1}$-torsion elements in $H^0\left(M^{m+1}_{h-m-1}\right)$. On the other hand, the Bockstein spectral sequence implies for any element $x\in H^0\left(M^{m+1}_{h-m-1}\right)$, there is a $k$ such that $v_{h-m+1}^kx\in H^0\left(M^{m}_{h-m}\right)$.  We can therefore obtain an additive basis for $H^0\left(M^{m+1}_{h-m-1}\right)$ from that for $H^0\left(M^{m}_{h-m}\right)$ by taking their quotients of powers of $v_{h-m+1}$. 
		
		Let $[x]\in H^0\left(M^{m+1}_{h-m-1}\right)$. It is can be divided by $v_{h-m+1}$ in $H^0\left(M^{m+1}_{h-m-1}\right)$  iff $\delta([x])= [0]$ in the long exact sequence above. Pick a representative cocycle $x$ for $[x]$. From the definition of the connecting homomorphism in long exact sequence, we know $\delta([x])$ is represented by the cocycle $d(\frac{x}{v_{h-m-1}})$, where $d$ is the cobar differential. This cocycle being zero in $H^1\left(M^{m}_{h-m}\right)$ means that $d(\frac{x}{v_{h-m-1}})=d(\varepsilon)$ for some correcting term $\varepsilon\in M^{m}_{h-m}$. Now set $x'=x-v_{h-m-1}\cdot\varepsilon$. Then $x'\equiv x\mod v_{h-m-1}$ and $x'$ can be divided by $v_{h-m-1}$ in $H^0\left(M^{m+1}_{h-m-1}\right)$. 
		
		Then the inductive hypothesis says $H^0\left( M^{m}_{h-m}\right)$ is generated by the three family of elements $\left\{\frac{v_h^s}{pv_1\cdots v_{h-1}}\right\}\cup \left\{\frac{1}{pv_1^{d_1}\cdots v_{h-1}^{d_{h-1}}}\right\}\cup \left\{\frac{y^s_{h,N}}{pv_1^{d_1}\cdots v_{h-1}^{d_{h-1}}}\right\}$. Apply the procedure above to those generators $[x]$ until $\delta([x]/v_{h-m-1}^k)\ne [0]\in H^1\left(M^{m}_{h-m}\right)$, we obtain an additive basis for $H^0\left(M^{m+1}_{h-m-1}\right)$. It remains to check the new basis obtained from Families I and II generators in $H^0\left( M^{m}_{h-m}\right)$ have the desired forms. For Family II, the claim follows from the cobar differential $d(1)=0$. 
		
		For Family I, we can compute the cobar differential using \cite[(6.1.13)]{green}
		\[\delta\left(\frac{v_h^s}{pv_1\cdots v_{h-m-1}v_{h-m}\cdots v_{h-1}}\right)=d\left(\frac{v_h^s}{pv_1\cdots v_{h-m-1}^2v_{h-m}\cdots v_{h-1}}\right)=\frac{sv_h^{s-1}t_{m+1}^{p^{h-m-1}}}{pv_1\cdots v_{h-1}}.\]
		This is a non-zero cocycle in $H^1(M^{m}_{h-m})$ by \cite[Theorem 6.5.12]{green}.\footnote{Note that the $h_{i,j}$ in the cited theorem is represented by the cocycle $t_i^{p^j}$.}  As a result, the zero cocycle $\left[\frac{v_h^s}{pv_1\cdots v_{h-1}}\right]$ is not $v_{h-m-1}$-divisible in $H^0\left(M^{m+1}_{h-m-1}\right)$. This proves the form of Family I elements. 
	\end{proof}
	\begin{rmk}\label{rem:correcting_terms}
		To get a full account of $H^0(M^{h-1}_1)$ using the method above, we will need to have knowledge of $H^0(M^{h-2}_2)$ and $H^1(M^{h-2}_2)$. This in terms requires the  knowledge of of $H^0(M^{h-3}_3)$, $H^2(M^{h-3}_3)$, and $H^3(M^{h-3}_3)$. In the end, we will need to know $H^*(M^0_h)$ for $0\le *\le h-1$ to compute $H^0(M^{h-1}_1)$. These groups are only the inputs of the Bockstein spectral sequences. We still need to compute the cobar differentials to determine the additive bases at each step. This is why getting an additive basis for $H^0(M^{h-1}_1)$ is out of reach using the current technology. 
		
		One particular technical point in this computation is to find the correcting terms $\varepsilon$ in the proof above. Without them, Baird's \Cref{lem:Baird} would have given us the full basis. For a particular computation where one has to add correcting terms, a classic example arises from the $v_1$-Bockstein spectral sequence
		\[
		H^*(M^0_2)\otimes \F_p[v_1]/(v_1^\infty)\implies H^*(M^1_1)
		\]
		for primes $p\geq 5$. For example, as shown in \cite{green} and \cite{MRW} (cf. \cite{S_E2} for another account) the class $\frac{v_2^{p^2}}{pv_1^{p^2+1}}$ in the $E_1$-page of the $v_1$-BSS is a permanent cycle and so detects a class in $H^0(M^1_1)$. However, the element it detects is 
		\[
		\frac{v_2^{p^2}}{pv_1^{p^2+1}} - \frac{v_2^{p^2-p+1}}{pv_1^2} - \frac{v_2^{-p}v_3^p}{pv_1}\in M^1_1.
		\]
	\end{rmk}
	We now analyze degrees of elements in the three families in $H^{0}(M^{h-1}_{1})$ and study the degrees of corresponding elements in $H^{h^2}(\G_h;\pi_*(E_h))$ under duality. In \textbf{Family I}, the degrees of elements are given by:
	\begin{equation}\label{eqn:deg_fam_I}
		\left|\frac{v_h^s}{pv_1\cdots v_{h-1}}\right|=s|v_h|-\sum_{i=1}^{h-1}|v_i|=s|v_h|+2h-\frac{|v_h|}{p-1}.
	\end{equation}
	\begin{prop}\label{prop:fam_I}
		Let $J\trianglelefteq \pi_0(E_h)$ be an open invariant ideal containing $p$, such that $v_h^{p^N}$ is invariant modulo $J$. Then the Family I element $\frac{v_h^s}{pv_1\cdots v_{h-1}}$ determines a copy of $\F_p$ in $H_c^{h^2}(\G_h;\pi_t(E_h)/J)$ via Gross-Hopkins duality \Cref{prop:Hh2_reduction} and the change-of-rings \Cref{thm:CoR}, where \begin{equation}\label{eqn:t_congruence}
			t\equiv -\left(s+\frac{p^N-1}{p-1}\right)|v_h|\mod p^N|v_h|.
		\end{equation}
		In particular, 
		\begin{itemize}
			\item Elements in Family I contribute to $H_c^{h^2}(\G_h;\pi_t(E_h)/p)$ only when $|v_h|$ divides $t$.
			\item Family I elements determine a copy of $\F_p$ in $H_c^{h^2}(\G_h;\pi_0(E_h)/p)$.
		\end{itemize}
	\end{prop} 
	\begin{proof}
		By \Cref{prop:Hh2_reduction} and \Cref{thm:CoR}, we have isomorphisms 
		\begin{align*}
			H_c^{h^2}(\G_h;\pi_t(E_h)/J)&\cong \left(H^0_c\left(\G_h;\left.E_{2h-t-\frac{p^N|v_h|}{p-1}}\right/ J\right)\right)^\vee\\&\cong \left(H^{0,2h-t-\frac{p^N|v_h|}{p-1}}(M_1^{h-1}/J')\right)^\vee,
		\end{align*}
		where $J'\trianglelefteq BP_*$ is an invariant ideal corresponding to $J$ as in \Cref{lem:inv_ideal_E_BP}. By construction, elements in Family I are in $H^{0,*}(M_1^{h-1}/J')$ for all $J'$. To prove the claim, we need to compare the degrees of Family I elements \eqref{eqn:deg_fam_I} and the target degree $2h-t-\frac{p^N|v_h|}{p-1}$ above. Notice the $BP_*BP$-comodule $M_1^{h-1}/J'$ is $p^N|v_h|$-periodic by assumption. Solving for $t$ in the residue equation:
		\[2h-t-\frac{p^N|v_h|}{p-1}\equiv  s|v_h|+2h-\frac{|v_h|}{p-1}\mod p^N|v_h|, \]
		we obtain the congruence relation for $t$ in \eqref{eqn:t_congruence}. In particular, the number $t$ is necessarily divisible by $|v_h|$. Solving for $s$ when $t=0$, we obtain the Famliy I element
		\[\frac{v_h^{mp^N}\cdot v_h^{-\frac{p^N-1}{p-1}}}{pv_1\cdots v_{h-1}}\in H^{0,2h-\frac{p^N|v_h|}{p-1}}(M_1^{h-1}) \]
		that contributes to a copy of $\F_p\subseteq H^{h^2}_c(\G_h;\pi_0(E_h)/J)$ for some $m$. The claims about $H_c^{h^2}(\G_h;\pi_t(E_h)/p)$ then follows by passing to the colimit.
	\end{proof}
	It follows that we can prove \eqref{eqn:E_2p-2_duality_reduction} and \eqref{eqn:E_0_duality_reduction} by showing elements in Families II and III do not contribute to $H_c^{h^2}(\G_h;\pi_0(E_h)/J)$ and $H_0(\G_h;\pi_{2p-2}(E_h)/J)$ for any open invariant ideal $J$ containing $p$.
	
	Now suppose an element $\frac{1}{pv_1^{d_1}\cdots v_{h-1}^{d_{h-1}}}$ in \textbf{Family II} determines a non-zero element in $H_c^{h^2}(\G_h;\pi_t(E_h)/J)$, where $v_h^{p^N}$ is invariant modulo $J$. Then we have
	\begin{align*}
		-\sum_{i=1}^{h-1}d_i|v_i|&\equiv 2h-\frac{p^N|v_h|}{p-1}-t &&\mod p^N|v_h|\\
		\implies t&\equiv 2h+\sum_{i=1}^{h-1}d_i|v_i|-\frac{p^N|v_h|}{p-1} &&\mod p^N|v_h|.
	\end{align*}
	To estimate the bounds for $t$, we use the following lemma. 
	\begin{lem}[Baird, {\cite[Lemma 7.6]{MRW}}]\label{lem:Baird} 
		Let $s_1, \ldots, s_h$ be a sequence of positive integers, and let $p^{e_i}$ be the largest power of $p$ dividing $s_i$. Then the sequence 
		\[
		p, v_1^{s_1}, \ldots, v_n^{s_n}
		\]
		is an invariant ideal if and only if $s_i\leq p^{e_{i+1}}$ for $1\leq i <n$.
	\end{lem}
	In our case $s_h=p^N$, so the largest possible values of $d_i$ is when $d_1=d_2=\cdots=d_{h-1}=p^N$. The smallest possible value is when all the $d_i$'s are $1$. From this we get:
	\begin{equation}\label{eqn:t_bound_baird}
		-\frac{(p^N-1)|v_h|}{p-1}\le t\le 2h(1-p^N)\mod p^N|v_h|.
	\end{equation}
	Thus we have proved the following result:
	\begin{prop}
		Elements in Family II contribute to $H_c^{h^2}(\G_h; \pi_t(E_h)/J)$ via Gross-Hopkins duality \Cref{prop:Hh2_reduction} and the change-of-rings \Cref{thm:CoR} only when $t$ satisfies \eqref{eqn:t_bound_baird}, where $v_h^{p^N}$ is invariant modulo $J$.
	\end{prop} 
	\begin{cor}\label{cor:fam_II}
		Elements in Family II do not contribute to $H_c^{h^2}(\G_h; \pi_0(E_h)/p)$ or $H_c^{h^2}(\G_h; \pi_{2p-2}(E_h)/p)$.
	\end{cor}
	\begin{proof}
		This is because the residue class of $t=0$ or $2p-2$ never falls into the bounds in \eqref{eqn:t_bound_baird}.
	\end{proof}
	Now it remains to analyze elements in \textbf{Family III}. 	When $h=2$, this was computed by Miller-Ravenel-Wilson in \cite{MRW}. In the next subsection, we will study the implications of their computations. Nevertheless, we can get some general bounds for the $d_i$'s that would imply the RHVC and vanishing of $\kappa_h$ when $2p-1=h^2$. 
	\begin{prop}\label{prop:fam_III}~
		\begin{enumerate}
			\item Elements in Family III do not contribute through Gross-Hopkins duality and the change-of-rings theorem to $H_c^{h^2}(\G_h;\pi_0(E_h)/p)$ if for all invariant ideals of the form $J=(p,v_1^{d_1},\cdots, v_{h-1}^{d_{h-1}},y^s_{h,N})$, we have
			\begin{equation}\label{eqn:div_bounds}
				\sum_{i=1}^{h-1}d_i|v_i|< \frac{p^N|v_h|}{p-1}-2h.
			\end{equation}
			\item Similarly, these elements do not contribute through Gross-Hopkins duality and the change-of-rings theorem to $H_c^{h^2}(\G_h;\pi_{2p-2}(E_h)/p)$ if for all invariant ideals of the form $(p,v_1^{d_1},\cdots, v_{h-1}^{d_{h-1}},y^s_{h,N})$, we have
			\begin{equation}\label{eqn:div_bounds_2p-2}
				\sum_{i=1}^{h-1}d_i|v_i|< \frac{p^N|v_h|}{p-1}-2h+2p-2.
			\end{equation}
		\end{enumerate}
	\end{prop}
	\begin{proof}
		Similar to the Family II cases, suppose an element $\frac{y^s_{h,N}}{pv_1^{d_1}\cdots v_{h-1}^{d_{h-1}}}$ in Family III corresponds to non-zero element in $H_c^{h^2}(\G_h;\pi_t(E_h)/J)$, where $v_h^{p^N}$ is invariant modulo $J$. Then we have
		\begin{align*}
			s|y_{h,N}|-\sum_{i=1}^{h-1}d_i|v_i|&\equiv 2h-\frac{p^N|v_h|}{p-1}-t &\mod p^N|v_h|\\
			\implies t&\equiv 2h+\sum_{i=1}^{h-1}d_i|v_i|-\frac{p^N|v_h|}{p-1} &\mod p^N|v_h|.
		\end{align*}
		We want to show $t$ cannot be congruent to $0$ or $2p-2$ from this residue equation. Similar to the Family II case, we have $d_i\ge 1$. From this, we get the same lower bound for $t$ as in \eqref{eqn:t_bound_baird}:
		\[t\ge 2h+\sum_{i=1}^{h-1}|v_i|-\frac{p^N|v_h|}{p-1}=\frac{(1-p^N)|v_h|}{p-1}. \]
		The right hand side of this inequality is greater than both $-p^N|v_h|$ and $-p^N|v_h|+2p-2$. The bounds \eqref{eqn:div_bounds} imply $t<0$ in the residue equation. The lower and upper bounds together show that $t\not\equiv 0$ in the residue equation. Similarly, we can show the other bound \eqref{eqn:div_bounds_2p-2} implies $t\not \equiv 2p-2$ in the residue equation.
	\end{proof}
	The analysis above yields: 
	\begin{prop}\label{prop:bounds_RHVC_kappah}~
		\begin{enumerate}
			\item Suppose $p-1\nmid h$. If the bounds \eqref{eqn:div_bounds} hold, then the RHVC is true.
			\item Suppose $2p-1=h^2$. If the bounds \eqref{eqn:div_bounds_2p-2} hold, then $\kappa_h=0$. In particular, the first bounds \eqref{eqn:div_bounds} imply both the RHVC and $\kappa_h=0$ in this case.
		\end{enumerate} 
	\end{prop}
	\begin{proof}
		In \Cref{prop:Hh2_reduction}, we showed there is an isomorphism of groups using the duality theorems:
		\[
		H_c^{h^2}(\G_h;\pi_t(E_h)/p)\cong \colim_{p\in J\trianglelefteq \pi_0(E_h)}H^0_c\left(\G_h;\left.\pi_{2h-t-\frac{p^N|v_h|}{p-1}}(E_h)\right/ J\right)^\vee,
		\]
		where $J\trianglelefteq \pi_0(E_h)$ ranges through all open invariant ideals containing $p$ and $v_h^{p^N}$ is invariant mod $J$. Recall:
		\begin{enumerate}
			\item Combining the Poincar\'e duality between homology and cohomology  \eqref{eqn:PD_homology} and the isomorphism above, we proved in \eqref{eqn:E_0_duality_reduction} the RHVC reduces to the computation:
			\[ H^0_c\left(\G_h;\left.\pi_{2h-\frac{p^N|v_h|}{p-1}}(E_h)\right/ J\right)=\F_p. \]
			\item By \Cref{prop:ev_inj}, $\kappa_h$ injects into $H_c^{h^2}(\G_h;\pi_{2p-2}(E_h))$ when $2p-1=h^2$. The latter is isomorphic to $H_c^{h^2}(\G_h;\pi_{2p-2}(E_h)/p)$ by \Cref{prop:Hh2_mod_p}. In \eqref{eqn:E_0_duality_reduction}, we concluded the vanishing of $\kappa_h$ would follow from
			\[  H^0_c\left(\G_h;\left.\pi_{2h-(2p-2)-\frac{p^N|v_h|}{p-1}}(E_h)\right/ J\right)=0.\]
		\end{enumerate}
		By the Change-of-Rings \Cref{thm:CoR}, the two degree-zero cohomology groups are identified with $\Ext$-groups of $BP_*BP$-comodule $BP_*/J'$ in the corresponding internal degrees. They can be further viewed as a subgroups of $H^{0,*}(M^{h-1}_1)$. So we need to show 
		\begin{equation*}
			H^{0,*}(M^{h-1}_1)=\left\{\begin{array}{clr}
				\F_p& *= 2h-\frac{p^N|v_h|}{p-1},& \text{for the RHVC;}\\
				0& *=2h-(2p-2)-\frac{p^N|v_h|}{p-1}, &\text{for $\kappa_h=0$.}
			\end{array}\right.
		\end{equation*}
		By \Cref{prop:families}, elements in $H^{0,*}(M^{h-1}_1)$ are classified into three families:
		\begin{itemize}
			\item \Cref{prop:fam_I} says elements in Family I contribute a copy of $\F_p$ to $H^{0,*}(M^{h-1}_1)$ when $*= 2h-\frac{p^N|v_h|}{p-1}$. They have no contribution when $*=2h-(2p-2)-\frac{p^N|v_h|}{p-1}$.
			\item \Cref{cor:fam_II} shows elements in Family II do not contribute to $H^{0,*}(M^{h-1}_1)$ when $*=2h-\frac{p^N|v_h|}{p-1}$ or $2h-(2p-2)-\frac{p^N|v_h|}{p-1}$.
			\item The two bounds \eqref{eqn:div_bounds} and \eqref{eqn:div_bounds_2p-2} in \Cref{prop:fam_III} would respectively imply Family III elements do not contribute to $H^{0,*}(M^{h-1}_1)$ when $*=2h-\frac{p^N|v_h|}{p-1}$ or $2h-(2p-2)-\frac{p^N|v_h|}{p-1}$.
		\end{itemize} 
		Combining the three families above, we conclude the two bounds \eqref{eqn:div_bounds} and \eqref{eqn:div_bounds_2p-2} in \Cref{prop:fam_III} would respectively imply
		\begin{align*}
			H^{0,2h-\frac{p^N|v_h|}{p-1}}(M^{h-1}_1)=\F_p &\implies  \text{RHVC},\\
			H^{0,2h-(2p-2)-\frac{p^N|v_h|}{p-1}}(M^{h-1}_1)=0 &\implies \kappa_h=0.
		\end{align*}
		As the first bound \eqref{eqn:div_bounds} is stronger than the second \eqref{eqn:div_bounds_2p-2}, it would imply both the RHVC and $\kappa_h=0$ when $2p-1=h^2$.
	\end{proof}
	\begin{rem}
		Baird's \Cref{lem:Baird} implies that elements in $H^{0,*}(M^{h-1}_1)$ with numerator $v_h^{sp^N}$ for some $N\ge 1$ and $(s,p)=1$ must be of the form:
		\begin{equation*}
			\frac{v_h^{spN}}{pv_1^{s_1}\cdots v_{h-1}^{s_{h-1}}},
		\end{equation*}
		such that the sequence $(s_1,\cdots,s_{h-1},sp^N)$ satisfies $s_{i}\le p^{v_p(s_{i+1})}$. It follows that the largest values of the $s_i$'s are $s_1=s_2=\cdots=s_{h-1}=p^N$. One can then check that
		\begin{equation*}
			\sum_{i=1}^{h-1}s_i|v_i|=p^N\sum_{i=1}^{h-1}|v_i|= p^N\left(\frac{2(p^h-1)}{p-1}-2h\right)=\frac{p^N|v_h|}{p-1}-p^N\cdot 2h
		\end{equation*}
		This is strictly smaller than both bounds \eqref{eqn:div_bounds} and \eqref{eqn:div_bounds_2p-2} since $N\ge 1$.  As is explained in \Cref{rem:correcting_terms}, we can add correcting terms in lower Bockstein filtrations to $v_h^{sp^N}$ to increase their $v_i$-divisibility for $1\le i\le h-1$.  This is why we cannot deduce from Baird's \Cref{lem:Baird} that the bounds \eqref{eqn:div_bounds} and \eqref{eqn:div_bounds_2p-2} are always satisfied .
	\end{rem}
	\subsection{Consequences of the Miller-Ravenel-Wilson computation}
	Recall that $M^1_{h-1}$ is defined to be $v_h^{-1}BP_*/(p,v_1,\cdots,v_{h-2},v_{h-1}^\infty)$. In this subsection, we discuss some consequences of the computations of $H^0(M^1_{h-1})$  in \cite{MRW} on the RHVC when $(p-1)\nmid h$ and the exotic Picard groups when $2p-1=h^2$. The computations at height $2$ are given by: 
	\begin{thm}[Miller-Ravenel-Wilson, {\cite[Theorem 5.3]{MRW}}]\label{thm:MRW_M11}
		\begin{align*}H^{0,*}(M^1_1)&\cong\F_p\left\{\left.\frac{v_2^s}{pv_1}\right| s\in \Z,p\nmid s  \right\}\bigoplus \F_p\left\{\left.\frac{1}{pv_1^j}\right| j\ge 1\right\}\\ &\quad \bigoplus\F_p\left\{\left.\frac{x_N^s}{pv_1^{e_1}}\right|N\ge 1, s\in \Z, p\nmid s, 1\le e_1\le p^N+p^{N-1}-1 \right\},  \end{align*}
		where $x_N$ is defined inductively by 
		\begin{align*}
			x_0&=v_2,\\x_1&=x_0^p-v_1^pv_2^{-1}v_3,\\
			x_2&=x_1^p-v_1^{p^2-1}v_2^{(p-1)p+1}-v_1^{p^2+p-1}v_2^{p^2-2p}v_3,\\
			x_N&= x_{N-1}^p-2v_1^{(p+1)(p^{N-1}-1)}v_2^{(p-1)(p^{N-1}+1)},\qquad N\ge 3.
		\end{align*}
		The internal degree of $x_N^s$ is $sp^N |v_2|-e_1 |v_1|$. 
	\end{thm}
	Using Gross-Hopkins duality \Cref{prop:Hh2_reduction}, the results above imply the top degree cohomology groups of $\G_2$ with coefficients in $\pi_t(E_2)/p$ are:
	\begin{prop}\label{prop:ht2_residue_equation}
		Let $[\alpha]\in H^4_c(\G_2;\pi_t(E_2)/p)$ be a non-zero cohomology class. If $[\alpha]$ corresponds to an element $\frac{x_N^s}{pv_1^{e_1}}\in  H^{0,*}(M^1_1)$ for some $N\ge 1$ via the Gross-Hopkins duality, then 
		\begin{equation*}
			t\equiv -\frac{(p^N-1)|v_2|}{p-1}+(e_1-1) |v_1|\mod p^N  |v_2|.
		\end{equation*} 
	\end{prop}
	\begin{proof}
		By assumption, the element $\frac{x_N^s}{pv_1^{e_1}}$ is in the image of $H^{0,sp^N|v_2|-e_1|v_1|}(M^1_1/J)$ for some $J$ containing $p$ where $BP_*/J$ has a $v_2^{p^N}$-self map. The Poincar\'e duality \eqref{eqn:PD_homology} gives an isomorphism:
		\[H^4_c(\G_2;\pi_t(E_2)/p)\cong H^0_c(\G_2; \pi_{4-t}(E_2)\ldetr/(p,u_1^\infty))^\vee. \]
		By \Cref{thm:det_mod_p}, the determinant twist mod $J$ is identified with: 
		\[\pi_{4-t}(E_2)\ldetr/J=\pi_{4-t}\left.\left( \Sigma^{\frac{p^N|v_2|}{p-1}}E_2\right)\right/J=\pi_{4-t-\frac{p^N|v_2|}{p-1}}(E_2)/J.\]
		The claim now follows by solving for $t$ in the residue equation: 
		\[4-t-\frac{p^N|v_2|}{p-1}\equiv sp^N|v_2|-e_1|v_1| \mod p^N|v_2|. \qedhere\]
	\end{proof}
	In this way, we have recovered the patterns of the top-degree cohomology $H^4_c(\G_2,\pi_t(E_2)/p)$ in the computation by Behrens in \cite[Figure 3.2]{S_E2} when $p\ge 5$.
	\begin{cor}\label{cor:H4_G2_mod_p}
		$H^4_c(\G_2;\pi_t(E_2)/p)\neq 0$ iff either $|v_2|$ divides $t$, or $|v_1|$ divides $t$ and there is an $N\ge 1$ such that
		\begin{align*}
			-\frac{(p^N-1)  |v_2|}{p-1}\le t&\le -\frac{(p^N-1)|v_2|}{p-1}+|v_1|(p^N+p^{N-1}-2)&&\mod p^N|v_2|\\&=-2p^N-2p^{N-1}-2p+6 &&\mod p^N|v_2|.
		\end{align*}
	\end{cor}
	\begin{proof}
		In degrees divisible by $|v_2|$, we have elements corresponding to $\frac{v_2^s}{pv_1}$. When $|v_2|\nmid t$, this follows from \Cref{prop:ht2_residue_equation} and the bounds for $e_1$ in \Cref{thm:MRW_M11}: $1\le e_1 \le p^N+p^{N-1}-1$.
	\end{proof}
	We have therefore recovered the following result of Shimomura and Yabe in \cite{Shimomura-Yabe}: 
	\begin{cor}
		The RHVC holds and $H^4_c(\G_2;\pi_{2p-2}(E_2))=0$ when $h=2$ and $p\ge 5$.
	\end{cor}
	\begin{rem}
		Shimomura and Yabe proved the cohomological version of \Cref{conj:CVC} at $h=2$ and $p\ge 5$, which is equivalent to the homological version by Poincar\'e duality \Cref{cor:PD}.
	\end{rem}
	\begin{proof}
		When $|v_2|\nmid t$, the upper bounds for $t$ above are always negative, which implies when $p\ge 5$
		\begin{align*}
			H_0(\G_2;\pi_0(E_2)/p)\cong H^4_c(\G_2;\pi_0(E_2)/p)&=\F_p,\\
			H_0(\G_2;\pi_{2p-2}(E_2))\cong H^4_c(\G_2;\pi_{2p-2}(E_2))\cong H^4_c(\G_2;\pi_{2p-2}(E_2)/p)&=0.
		\end{align*}
		We have therefore verified \eqref{eqn:RHVC} and the vanishing of the top degree cohomology group $H^4_c(\G_2;\pi_{2p-2}(E_2))$.
	\end{proof}
	At height $h\ge 3$, $H^0(M^1_{h-1})$ is described as follows:
	\begin{thm}[Miller-Ravenel-Wilson, {\cite[Theorem 5.10]{MRW}}]\label{thm:MRW_M1_h-1}
		Define $a_{h,N}$ by the recursive formula: $a_{h,0}=1$, $a_{h,1}=p$, and 
		\begin{equation*}
			a_{h,N}=\left\{\begin{array}{cl}
				pa_{h,N-1},&1<N\not\equiv 1\mod(h-1);\\
				pa_{h,N-1}+p-1,&1<N\equiv 1\mod(h-1).
			\end{array}\right.
		\end{equation*}				
		Recall $M^1_{h-1}=v_h^{-1}BP_*/(p,v_{1},\cdots, v_{h-2}, v_{h-1}^\infty)$. Then $H^0(M^1_{h-1})$ is an $\F_p$-vector space generated by
		\begin{enumerate}[label=\textup{\Roman*.},itemsep=1 ex]
			\item  $\frac{v_h^s}{pv_1\cdots v_{h-1}}$, where $p\nmid s\in \Z$.
			\item  $\frac{1}{pv_1\cdots v_{h-2}v_{h-1}^{j}}$, where $j\ge 1$.
			\item $\frac{x_{h,N}^s}{pv_1\cdots v_{h-2}v_{h-1}^{e_{h-1}}}$, where $p\nmid s\in \Z$, $1\le e_{h-1}\le a_{h,N}$, and $x_{h,N}$ is  is defined inductively by 
			\begin{align*}
				x_{h,0}&=v_p,&&\\
				x_{h,1}&=v_h^p-v_{h-1}^pv_{h}^{-1}v_{h+1},&&\\
				x_{h,N}&=x_{h,N-1}^p && \text{ for }1<N\not\equiv 1\mod(h-1),\\
				x_{h,N}&= x_{h,N-1}^p-v_{h-1}^{\frac{(p^{N-1}-1)(p^h-1)}{p^{h-1}-1}}v_{h}^{p^N-p^{N-1}+1}&& \text{ for }1<N\equiv 1\mod(h-1).
			\end{align*}
		\end{enumerate}
	\end{thm}
	\begin{lem}\label{lem:a_hn_closed_formula}
		The closed formula of $a_{h,N}$ is given by:
		\begin{equation*}
			a_{h,N}=p^N+\frac{(p-1)(p^{N-1}-p^{r-1})}{p^{h-1}-1},
		\end{equation*}
		where $1\le r\le h-1$ is an integer such that $N\equiv r\mod (h-1)$. \footnote{$r$ is \emph{not} the usual residue of $N$ mod $h-1$ since $r=h-1$ when $(h-1)\mid N$.}
	\end{lem}
	Like \Cref{cor:H4_G2_mod_p}, we now have:
	\begin{prop}\label{prop:MRW_vanishing}
		Assume $(p-1)\nmid h$ and let $I_{h-1}=(p,u_1,\cdots,u_{h-2})\trianglelefteq \pi_0(E_h)$. Then the cohomology group $H_c^{h^2}(\G_h;\pi_t(E_h)/I_{h-1})$ is zero unless $|v_h|$ divides $t$, or there is an $N\ge 1$ such that
		\[t\equiv -\frac{(p^N-1)|v_h|}{p-1}+k\cdot|v_{h-1}| \mod p^N|v_h|\text{ for some } 0\le k\le a_{h,N}-1.\]
	\end{prop}
	In particular, the closed formula for $a_{h,N}$ in \Cref{lem:a_hn_closed_formula} implies the upper bounds for $t$ above are always negative. Like the $h=2$ and $p\ge 5$ case in \Cref{cor:H4_G2_mod_p}, this shows that when $(p-1)\nmid h$: 
	\begin{align}
		H^{h^2}_c(\G_h;\pi_0(E_h)/I_{h-1})&=\F_p,\nonumber\\
		H^{h^2}_c(\G_h;\pi_{2p-2}(E_h)/I_{h-1})&=0.\label{eqn:Hh2_mod_I_h-1}
	\end{align}
	\begin{thm}[Main Theorem B]\label{thm:RHVC_Ih-2}
		When $(p-1)\nmid h$, the Homological Vanishing Conjecture is true modulo the ideal $I_{h-1}=(p,u_1,\cdots,u_{h-2})$.
	\end{thm}
	\subsection{Conclusions at small heights and primes}
	Recall that by \Cref{thm:DSS_pic}, there is an isomorphism when $2p-1=h^2$: 
	\begin{equation*}
		\begin{tikzcd}
			\kappa_h\rar[>->,"{\sim}","\text{(\ref{prop:ev_inj})}"']& H^{2p-1}_c(\G_h;\pi_{2p-2}(E_h))\rar["\sim","\text{(\ref{prop:Hh2_mod_p})}"']&H^{h^2}_c(\G_h;\pi_{2p-2}(E_h)/p).
		\end{tikzcd}
	\end{equation*}
	At $p=5$ and $h=3$, to use our method to compute  $H^{9}_c(\G_3;\pi_8(E_3)/5)$, we need to know $H^{0,*}(M^2_1)$ at prime $p=5$. It is also needed to verify the RHVC at height $h=3$ and $p>2$ (which implies $(p-1)\nmid h$).  This computation also appears in Yexin Qu's  thesis \cite{qu_2018}. By \Cref{prop:bounds_RHVC_kappah}, we need to check that for each $1\le e_2\le a_{3,N}$, if there is element $\frac{y_{N}}{pv_1^{e_1}v_2^{e_2}}\in H^0(M^{2}_1)$, then
	\begin{equation*}
		e_1\cdot |v_1|+e_2\cdot |v_2| < \frac{p^N|v_3|}{p-1}-2\cdot 3.
	\end{equation*}
	When $e_2=1$, we have $e_1< \frac{p^N(p^2+p+1)-3}{p-1}-(p+1)$. When $e_2$ attains its maximum $a_{3,N}$ in \Cref{thm:MRW_M1_h-1}, this translates to 
	\begin{equation*}
		e_1< \frac{p^{N-1}(p^2+p+1)-3}{p-1}+p^{r-1}, \quad r=\left\{\begin{array}{ll}
			1, &N \text{ is odd;}\\
			2, & N \text{ is even}.
		\end{array}\right.
	\end{equation*}
	We observe that both bounds are larger (looser) than the bounds $a_{3,N}$ for $v_2$-divisibility itself. However, it is not clear how to verify them without computing the Greek letter elements in $H^0(M^{2}_1)$. Nevertheless, the vanishing result in \eqref{eqn:Hh2_mod_I_h-1} does have concrete implications on exotic elements in $\Pic_{K(h)}$ when $2p-1=h^2$, provided the relevant Smith-Toda complexes exist. 
	\begin{thm}[Main Theorem A]\label{thm:main_A}
		Let $2p-1=h^2$. Suppose the type-$(h-1)$ Smith-Toda complex $V(h-2)=S^0/(p,v_1,\cdots, v_{h-2})$ exists at prime $p$. Then an exotic element $X\in \kappa_h$ cannot be detected by $V(h-2)$; that is, 
		\begin{equation*}
			X\wedge_{K(h)}V(h-2)\simeq L_{K(h)}V(h-2).
		\end{equation*}
	\end{thm}
	\begin{proof}
		Using the topology of $\Pic_{K(h)}$ described in \cite[Proposition 14.3.(d)]{Hovey-Strickland_1999}, we know that if the image of $X\in \kappa_h$ under the composite
		\begin{equation*}
			\kappa_3\xrightarrow{\ev_2}H^{h^2}_c(\G_h;\pi_{2p-2}(E_h))\twoheadrightarrow H^{h^2}_c(\G_3;\pi_{2p-2}(E_h)/I_{h-1})
		\end{equation*}
		is zero, then $X\wedge_{K(h)}V(h-2)=L_{K(h)}V(h-2)$, provided $V(h-2)=S^0/(p,v_1,\cdots, v_{h-2})$ exists. Since the target of this map is zero by \eqref{eqn:Hh2_mod_I_h-1}, the equivalence above is true for any $X\in \kappa_h$ when $2p-1=h^2$.
	\end{proof}
	\begin{cor}\label{cor:fin_complex}~
		\begin{enumerate}
			\item At height $3$ and prime $5$, an exotic element $X$ in $\Pic_{K(3)}$ cannot be detected by $V(1)=S^0/(5,v_1)$.
			\item  At height $5$ and prime $13$, an exotic element $X$ in $\Pic_{K(5)}$ cannot be detected by $V(3)=S^0/(13,v_1,v_2,v_3)$.
		\end{enumerate} 
	\end{cor}
	\begin{proof}
		The Smith-Toda complexes $V(1)$ and $V(3)$ have been constructed for $p\ge 3$ and $p\ge 7$ by Adams-Toda and Smith-Toda, respectively \cite[Example 2.4.1]{orange}.
	\end{proof}
	\begin{rem}\label{rem:ST_cpx}
		A referee has pointed out to us that it is an open question whether $V(4)$ exists any \emph{any} prime (see discussions at the end of \cite[\S5.6]{green}). Recall that Smith-Toda complexes $V(n)$ are constructed as cofibers of $v_n$-self maps of $V(n-1)$ that induce multiplication by $v_n$ on $BP$-homology groups. This means that we do not know the existence of $V(n)$ for $n\ge 4$ at any prime $p$. As a result, it is unclear whether we have a similar statement at the next pair of height and prime $(h,p)=(9,41)$ satisfying $2p-1=h^2$, which would require the existence of $V(7)$ at the prime $p=41$. 
		
		In \cite{Nave_Smith-Toda}, Nave proved the non-existence of the Smith-Toda complex $V(h)$ when $2h=p+1$. This does not overlap with our consideration of the potential Smith-Toda complexes $V(h-2)$ when $h^2=2p-1$.
	\end{rem}
	\begin{rem}\label{rem:fin_cpx}
		By \cite[Corollary 7.11]{Hovey-Strickland_1999}, a $K(h)$-local spectrum $X$ is equivalent to $L_{K(h)}S^0$ iff $X\wedge_{K(h)}V\simeq L_{K(h)} V$ for all finite complexes of type $h$. This means if $X\wedge_{K(h)} V\simeq  L_{K(h)} V$ for all $X\in \kappa_h$ and finite complexes $V$ of type $n$, then $\kappa_h=0$. \Cref{thm:main_A} can be thought of as a first step towards showing $\kappa_h=0$ when $2p-1=h^2$, since it implies $X\wedge_{K(h)} V\simeq  L_{K(h)}  V$ for any cofibers $V$ of $v_h$-self maps of $V(h-2)$. Our choices of finite complexes are restricted to cofibers of the Smith-Toda complexes $V(h-2)$, because we do not have better Greek letter element computation results beyond \Cref{thm:MRW_M1_h-1} in \cite{MRW} when $h\ge 3$. 
	\end{rem}
	We can also use the same technique to study the subgroup $\kappa^{(1)}_{h}$ of $\kappa_h$ when $4p-3=h^2$.  Recall from \eqref{eqn:tower}, $\kappa_h^{(1)}$ is the kernel of detection map
	\begin{equation*}
		\ev_2\colon \kappa_h\longrightarrow H^{2p-1}_c(\G_h;\pi_{2p-2}(E_h)).
	\end{equation*}
	In terms of the homotopy fixed point spectral sequence, it consists of exotic $K(h)$-local spheres $X$, such that $E_2^{0,0}(X)\cong \Zp$ does not support a $d_{2p-1}$-differential. Using similar argument as in \Cref{prop:ev_inj}, one can show that the detection map:
	\begin{equation*}
		\ev_3\colon \kappa_h^{(1)}\longrightarrow E_{2p}^{4p-3,4p-4}
	\end{equation*}
	injective because the target of the next detection map is above the horizontal vanishing line at $s=h^2=4p-3$ of the $E_2$-page. The target of this detection map is a subquotient of 
	\begin{equation*}
		E_{2}^{4p-3,4p-4}=H^{4p-3}_c(\G_h;\pi_{4p-4}(E_h))=H^{h^2}_c(\G_h;\pi_{4p-4}(E_h)) . 
	\end{equation*}
	By \Cref{prop:MRW_vanishing}, we know $H^{h^2}_c(\G_h;\pi_{4p-4}(E_h)/I_{h-1})=0$ when $(p-1)\nmid h$. This implies: 
	\begin{thm}\label{thm:kappa_h_4p-3}
		Let $X$ be an exotic element in $\Pic_{K(h)}$ where $h$ and $p$ satisfies $4p-3=h^2$.  Suppose the Smith-Toda complex $V(h-2)$ exists. If $X\in \ker\ev_2$, i.e. the $E_{2}^{0,0}(X)$-term in the HFPSS \eqref{eqn:HFPSS_X} does not support a $d_{2p-1}$-differential, then $X\wedge_{K(h)} V(h-2)\simeq L_{K(h)}V(h-2)$. In particular, this is true when $(h,p)=(3,3)$ and $(h,p)=(5,7)$. 
	\end{thm}
	We end this paper with a discussion on the relation between the RHVC and exotic Picard groups. 
	\begin{thm}[Main Theorem C]\label{thm:main_C}
		At height $3$, the RHVC implies $\kappa_3=0$ when $p=5$ and $\kappa_3^{(1)}=0$ when $p=3$.
	\end{thm}
	\begin{proof}
		We will prove the contra-positive statement at $p=5$ first. Suppose $\kappa_3\ne 0$ at $p=5$. By \Cref{prop:ev_inj} and \Cref{prop:Hh2_mod_p}, we know $H^9_{c}(\G_3;\pi_8(E_3)/5)\ne 0$. Let $x$ be a nonzero element in this group. Under the isomorphism in \Cref{prop:Hh2_reduction}, $x$ corresponds to a family of non-zero elements \eqref{eqn:E_2p-2_duality_reduction}
		\[\xi_J\in H^0_c\left(\G_3;\left.\pi_{2\cdot 3-(2\cdot 5-2)-\frac{5^N(2\cdot 5^3-2)}{5-1}}(E_{3})\right/J\right)\]
		for cofinal system of open invariant ideals $J$ in $\pi_0(E_3)$ that contains $5$. By \Cref{prop:MRW_vanishing}:
		\[H^0_c\left(\G_3;\left.\pi_{2\cdot 3-(2\cdot 5-2)-\frac{5^N(2\cdot 5^3-2)}{5-1}}(E_{3})\right/(5,v_1,v_2^\infty)\right)=0,\]
		which implies the element $\xi_J$ cannot be $v_1$-torsion. By \Cref{prop:fam_I} and \Cref{cor:fam_II}, the $\xi_J$'s are necessarily Family III Greek letter elements in \Cref{prop:families}. As result, we obtain a compatible family of non-zero Family-III elements 
		\[\xi'_J=v_1\alpha_J\in H^0_c\left(\G_3;\left.\pi_{2\cdot 3-\frac{5^N(2\cdot 5^3-2)}{5-1}}(E_{3})\right/J\right).\]
		Again by \Cref{prop:Hh2_reduction}, $\xi'_J$ corresponds a non-zero element  $x'\in H^9_{c}(\G_3;\pi_0(E_3)/5)$. Recall from \Cref{prop:fam_I}, this group already has a copy of $\F_5$ coming from Family I elements through Gross-Hopkins duality. The new addition of $x'$ in this group from Family III elements shows that its dimension is at least $2$, which contradicts the RHVC. 
		
		At $p=3$, we know $\kappa_3^{(1)}$ injects into the $E_{2p}^{4p-3,4p-4}$-term in the HFPSS for the $K(3)$-local sphere. If $\kappa_3^{(1)}\ne0$, then neither is $E_{2p}^{4p-3,4p-4}=E_6^{9,8}$. This implies $E_2^{9,8}=H^9_c(\G_3;\pi_8(E_3))\ne 0$, since $E_6^{9,8}\ne 0$ is its subquotient. The rest of the argument is entirely the same as the $p=5$ case. 
		
		In this way, we conclude $\kappa_3\ne 0$ at $p=5$  and $\kappa^{(1)}_3=0$ at $p=3$ implies the RHVC is false at the respective primes. These are the contra-positive statements of the theorem. 
	\end{proof}
	\begin{rem}
		This proof relies on \Cref{prop:MRW_vanishing}, a consequence of the Miller-Ravenel-Wilson computation \Cref{thm:MRW_M1_h-1}. In general, the implication would hold at height $h$ if we knew 
		\begin{equation}\label{eqn: RHVC to kappah trivial in general}
			H^{0,2h-(2p-2)-\frac{p^N|v_h|}{p-1}}(M^{h-2}_2)=0
		\end{equation} 
		for all $N$. Miller-Ravenel-Wilson have calculated $H^{0,*}(M_{h-1}^1)$ for all $h$. To prove \eqref{eqn: RHVC to kappah trivial in general} one would have to calculate $h-3$ many Bockstein spectral sequences, which seems dizzyingly beyond our reach with current technology.  
	\end{rem}

	\bibliographystyle{plain}
	\bibliography{exoticPicard.bib}
\end{document}